\documentclass[3p,times]{elsarticle}

\usepackage{amssymb}
\usepackage{amsmath}
\usepackage{amsthm}
\usepackage{braket,amsfonts}
\usepackage{color}

\newtheorem{theorem}{Theorem}
\newtheorem{remark}{Remark}
\newtheorem{lemma}{Lemma}
\newtheorem{proposition}{Proposition}

\usepackage[numbers]{natbib}

\usepackage[figuresright]{rotating}

\usepackage{amsopn}
\usepackage{enumitem}

\setlist[enumerate]{label=$\rm{(\roman*)}$,leftmargin=\parindent}

\newcommand{\R}{{\mathbb R}}

\DeclareMathOperator{\argmin}{argmin}
\DeclareMathOperator{\prox}{prox}

\newcommand{\cE}{{\mathcal E}}
\newcommand{\cH}{{\mathcal H}}
\newcommand{\cO}{{\mathcal O}}

\newcommand{\demi}{\frac{1}{2}}

\newcommand{\rinf}{\R\cup\{+\infty\}}
\newcommand{\eqdef}{:=}

\newcommand{\norm}[1]{\left\|{#1}\right\|}
\newcommand{\pa}[1]{\left({#1}\right)}

\DeclareMathOperator{\dom}{dom}
\DeclareMathOperator{\Zer}{Zer}

\def\s{\sigma}
\def\l{\lambda}
\def\<{\langle}
\def\>{\rangle}

\begin{document}

\begin{frontmatter}




\title{Fast convex optimization via time scale and averaging of the  steepest descent}


%

\author[1]{Hedy Attouch\corref{cor1}%
}
\ead{hedy.attouch@umontpellier.fr}
\author[2]{Radu Ioan Bo\c{t}\fnref{fn1,fn2}}
\ead{radu.bot@univie.ac.at}
\author[2]{Dang-Khoa Nguyen\fnref{fn2}}
\ead{dang-khoa.nguyen@univie.ac.at}
\cortext[cor1]{Corresponding author}
\fntext[fn1]{The research of RIB has been supported by FWF (Austrian Science Fund), projects W 1260 and P 34922-N.}
\fntext[fn2]{The research of DKN has been supported by FWF (Austrian Science Fund), projects W 1260.}
\affiliation[1]{organization={IMAG, Univ. Montpellier},
	addressline={CNRS},
	postcode={},
	city={Montpellier},
	country={France}}
\affiliation[2]{organization={Faculty of Mathematics, University of Vienna},
	addressline={Oskar-Morgenstern-Platz 1},
	postcode={1090 Vienna},
	city={Vienna},
	country={Austria}}

\begin{abstract}
In a  Hilbert setting, we develop a  gradient-based dynamic approach for fast solving convex optimization problems. 
By applying time scaling, averaging, and perturbation techniques to the continuous steepest descent (SD), we obtain high-resolution ODEs of the Nesterov and Ravine methods. These dynamics involve asymptotically vanishing viscous damping and Hessian driven damping (either in explicit or implicit form).  Mathematical analysis does not require developing a Lyapunov analysis for inertial systems. We simply exploit classical convergence results for (SD) and its external perturbation version, then use tools of differential and integral calculus, including Jensen's inequality. The method is flexible and by way of illustration we show how it applies starting from other important dynamics in optimization. We consider the case where the initial dynamics is the regularized Newton method, then the case where the starting dynamics is the differential inclusion associated with a convex lower semicontinuous potential, and finally we show that the technique can be naturally extended to the case of a monotone cocoercive operator. Our approach leads to parallel algorithmic results, which we study in the case of fast gradient and proximal algorithms. Our averaging technique shows new links between the Nesterov and Ravine methods.
\end{abstract}

\begin{keyword}

fast convex optimization \sep 
damped inertial dynamic \sep 
time scaling \sep 
averaging \sep 
Nesterov and Ravine algorithms \sep 
Hessian driven damping \sep 
proximal algorithms

\MSC[2020] 37N40 \sep 46N10 \sep 49M30 \sep 65B99 \sep 65K05 \sep 65K10 \sep 90B50 \sep 90C25
\end{keyword}

\end{frontmatter}


\section{Introduction}\label{sec:prel} 
In a real Hilbert space $\cH$, we develop a general dynamic approach
whose objective is the rapid resolution of convex optimization problems via first-order methods.
We consider the   problem
\begin{equation}\label{basic-min}
\min \left\lbrace  f(x) : \, x\in\cH  \right\rbrace,
\end{equation}
where, throughout the paper, we make the following  assumptions on the function $f$ to be minimized
\begin{equation}\label{eq:basic_hypo}
(\mathcal A) \ \begin{cases}
	f: \cH \rightarrow \R  \text{ is a convex function of  class }    {\mathcal C}^1; \, S= \argmin_{\cH} f \neq \emptyset ;
	\vspace{1mm}\\
	\nabla f   \text{ is Lipschitz continuous on the bounded sets of } \cH.
\end{cases}
\end{equation}
Our study is part of the close links between dissipative dynamical systems and optimization algorithms, the latter being obtained by temporal discretization of the continuous dynamics.
The dynamic system approach makes it possible to harmoniously combine the  ingredients of our approach, namely  time scaling, averaging, and perturbation techniques to finally obtain inertial dynamics with proven fast optimization properties.

Let us recall the mainstream for the study of accelerated gradient methods. Then, in the next section, we will describe our approach, which is novel in many aspects. The classical approach can be traced back to Gelfand and Tsetlin \cite{GT}, then Polyak \cite{Polyak1,Polyak2}, and Attouch, Goudou and Redont \cite{AGR}, where fast first-order optimization algorithms are naturally linked to damped inertial dynamics via temporal discretization. In  this approach, which is very inspired by mechanics, PDE's, and control theory  the focus is on the design of the damping term of the inertial dynamic, and its asymptotic stabilization effect.

In recent years, for the minimization of general convex differentiable functions, an in-depth study has been carried out linking the accelerated  gradient method of Nesterov \cite{Nest1,Nest2} to inertial dynamics with vanishing viscosity (the viscous damping coefficient tends to zero as $t\to +\infty$). That is the Su-Boyd-Cand\`es dynamical system \cite{SBC}
\begin{equation}\label{SBC}
\ddot{x}(t) + \frac{\alpha}{t} \dot{x}(t) + \nabla f (x(t)) =0,
\end{equation}
where fast optimization is obtained by taking $\alpha \geq 3$.
In addition to  viscous damping, taking into account geometric damping (in our situation it is driven by the Hessian of the function to be minimized) makes it possible to improve the performance of these methods by attenuating the oscillations. That is the dynamic
\begin{equation}\label{Hessian_damping}
\ddot{x}(t) + \frac{\alpha}{t} \dot{x}(t) +  \beta  \nabla^2  f (x(t)) \dot{x} (t)  + \nabla f (x(t)) =0,
\end{equation}
considered by  Attouch, Peypouquet, and  Redont in \cite{APR2}.
The central role played by these dynamics has been confirmed by their recent  interpretation as the low and high resolution of the Nesterov and Ravine accelerated gradient methods, see Shi, Du, Jordan and Su \cite{SDJS},  Attouch, Chbani, Fadili, and Riahi \cite{ACFR,ACFR-Opti}.
The study of the fast convergence properties of these dynamics and algorithms is based on Lyapunov analysis.
But this approach  does not give a simple explanation of the choice of dynamics and algorithms. Moreover the Lyapunov analysis is not trivial, and there is no general rule to find the Lyapunov functions.
Still, these techniques have allowed a deep understanding of dynamics and accelerated algorithms for optimization,
see \cite{ABCR, ABC, AC1,AC2, ACFR, ACPR, ACR-rescale, AF, AP, BCL, Bot-Nguyen, CEG1, CBFP, CD, SBC, SDJS} for some recent contributions.

\section{Time scaling and averaging approach}
We propose a new approach to these questions which is based on a combination of time scaling, averaging, and perturbation techniques.
The basis property which underlies our approach is the fact that by time scaling one can always accelerate a dynamic. The kinematic interpretation is clear. By changing the time clock the trajectories are travelled more or less quickly.
The averaging technique shifts from a dynamic of order $n$ to a dynamic of the immediately higher order $n+1$. While preserving the asymptotic convergence properties, this makes it possible to achieve a dynamic where the coefficient in front of the gradient is fixed, and therefore suitable for developing an accelerated gradient algorithm. 
The perturbation technique gives flexibility to these methods,  exploiting the fact that the convergence properties are preserved  when adding an external perturbation which vanishes asymptotically sufficiently fast.

As the basic starting dynamic used in this paper, we consider the classical continuous steepest descent
\begin{equation}\label{SD}
{\rm (SD)} \quad   \dot{z}(t) + \nabla f (z(t)) =0.
\end{equation}
Under the standing assumption $(\mathcal A)$ on $f$  we know that, for any $z_0 \in \cH$, there exists a unique classical global solution $z\in \mathcal C^1([t_0, +\infty[: \cH )$  of (SD) satisfying  $z(t_0 )= z_0$, see \cite[Theorem 17.1.1]{ABM_book}. We fix $t_0 >0$ as the origin of time (because of the singularity at the origin of  coefficient $\frac{\alpha}{t} $ which enter into our analysis).

\subsection{Classical facts concerning the continuous steepest descent} 

We will have to consider the gradient system with an external perturbation. The following theorem investigates its asymptotic properties under appropriate conditions on the perturbation term. The proof is provided in the Appendix.
\begin{theorem}\label{theorem SD pert}
Suppose that  $f \colon \cH \to \R$ satisfies $(\mathcal A)$.
Let $z \colon \left[ t_{0} , + \infty \right[ \to \cH$ be a solution trajectory of 	
\begin{equation}\label{pert SD}	
	\dot{z}(t) + \nabla f( z(t)) = g(t),
\end{equation}
where  $g \colon \left[ t_{0} , + \infty \right[ \to \cH$ is such that
\begin{equation}\label{pert 01}
	\int_{t_{0}}^{+ \infty} \left\lVert g \left( t \right) \right\rVert dt < + \infty  \textrm{ and } \int_{t_{0}}^{+ \infty} t \left\lVert g \left( t \right) \right\rVert ^{2} dt < + \infty .
\end{equation}
Then the following statements are true:
\begin{enumerate}
	\item \label{SD pert:int-vel} 
	(integral estimate of the velocities)
	$
	\int_{t_0}^{+\infty} t \left\lVert \dot{z} \left( t \right) \right\rVert ^{2} dt < + \infty;
	$
	\item \label{SD pert:int-grad} 
	(integral estimate of the gradients)
	$
	\int_{t_0}^{+\infty} t \left\lVert \nabla f \left( z \left( t \right) \right) \right\rVert ^{2}  dt < +\infty;
	$		
	\item \label{SD pert:int-fun} 
	(integral estimate of the values)
	$
	\int_{t_0}^{+\infty} \left( f \left( z \left( t \right) \right) -\inf_{\cH} f \right) dt < +\infty;
	$
	
	\item \label{SD pert:rate-fun} 
	(convergence of values towards minimal value)
	$
	f \left( z \left( t \right) \right) -\inf_{\cH} f = o \left( \frac{1}{t} \right) \textrm{ as } t \to + \infty;
	$
	
	\item \label{SD pert:traj}
	the solution trajectory $z(t)$ converges weakly as $t \to +\infty$, and its limit belongs to $S=\argmin f$.
	
	If
	\begin{equation*}
		\int_{t_{0}}^{+ \infty} t^{2} \left\lVert g \left( t \right) \right\rVert ^{2} dt < + \infty,
	\end{equation*}
	then 
	\item \label{SD pert:rate-grad}
	(convergence of gradients towards zero)	\quad 
	$
	\left\lVert \nabla f \left( z \left( t \right) \right) \right\rVert = o \left( \frac{1}{t} \right) \textrm{ as } t \to + \infty.
	$
\end{enumerate}
\end{theorem}


\subsection{General time scaling and averaging}\label{sec:time_scale}
Let us make the change of time variable $t=\tau (s)$ in the dynamic (SD),
\begin{equation}
{\rm(SD)} \qquad \dot{z}(t) + \nabla f( z(t)) = 0,
\end{equation}
where $\tau (\cdot)$ is an increasing  function from $\mathbb R^+$ to $\mathbb R^+$, which is continuously differentiable and satisfies $\lim_{s \to +\infty}\tau (s) = + \infty$. 
Set $y(s):= z(\tau(s))$ and $s_{0}$ be such that $t_{0} = \tau \left( s_{0} \right)$.
On the one hand, by the derivation chain rule, we have
\begin{equation}\label{change var01}
\dot{y} (s)= \dot{\tau}(s) \dot{z}(\tau(s)).
\end{equation}
On the other hand, setting $t=\tau (s)$ in (SD) gives
\begin{equation}\label{change var1}
\dot{z}(\tau(s))  + \nabla f(z(\tau(s))) =0.
\end{equation}
According to (\ref{change var01}) and (\ref{change var1}), we obtain
\begin{equation}\label{change var2-b}
\dot{y} (s)      +   \dot{\tau}(s) \nabla f(y(s)) =0.
\end{equation}
The convergence rate becomes (take $t=\tau(s)$ in \ref{SD pert:rate-fun} of Theorem \ref{theorem SD pert} with $g \equiv 0$)
\begin{equation}\label{SD_rescale_1}
f (y(s))- \inf_{\cH} f = o \left( \dfrac{1}{\tau(s)} \right) \textrm{ as } s \to + \infty.
\end{equation}

By introducing a function $\tau(s)$ that grows faster than the identity (namely $\tau(s) \geq s$), we have accelerated the dynamic, passing from the asymptotic convergence rate $1/t$ for (SD) to $1/\tau(s)$ for \eqref{change var2-b}.
The price to pay is that we no longer have an autonomous dynamic in \eqref{change var2-b}, with as major drawback the fact that the coefficient in front of the gradient term tends to infinity as $s \to +\infty$. 
This prevents from using gradient methods to discretize it. Recall that for gradient methods the step size has to be less than or equal to twice the inverse of the Lipschitz constant of the gradient. To overcome this we come with the second step of our method which is averaging.

Several recent papers have been devoted to the role of the  averaging techniques in the acceleration of optimization algorithms. In \cite{Nest3} Nesterov designed an algorithm that makes use of an averaging step for the gradients of the generated sequence. In \cite{SAB} Scieur, D'Aspremont, and Bach developed a universal acceleration method based on averaging via polynomial techniques. In \cite{PL} Poveda and Li developed another  universal acceleration method  via averaging technique for singularly perturbed hybrid dynamical systems with restarting mechanisms.
Let us explain some simple  and, as far as we know, new ideas concerning the averaging methods based on  differential equations. The use of differential equations allows us to use differential and integral calculus, which turns out to be a flexible approach.

%
Let us attach to $y(\cdot)$ the new function $x: [s_{0}, +\infty[ \to \cH$ defined by 
\begin{equation}\label{change var24}
\dot{x}(s) + \frac{1}{\dot{\tau}(s)}(x(s)-y(s))  = 0 ,
\end{equation}
with $x(s_{0}) =x_0$ given in $\cH$.
Equivalently
\begin{equation}\label{change var240}
y(s)=   x(s)+  \dot{\tau}(s) \dot{x}(s)  .
\end{equation}

\noindent By temporal derivation of  \eqref{change var240} we get
\begin{equation}\label{change var25}
\dot{y}(s) =  \dot{x}(s)+ \ddot{\tau}(s) \dot{x}(s)  + \dot{\tau}(s) \ddot{x}(s)   .
\end{equation}

\noindent Replacing $\dot{y}(s)$  as given by \eqref{change var25} in 
\eqref{change var2-b} we get
\begin{equation}\label{change var26}
\dot{\tau}(s) \ddot{x}(s)   + (1+ \ddot{\tau}(s)) \dot{x}(s) + \dot{\tau}(s)\nabla f(y(s))      = 0 .
\end{equation}
After simplification we obtain
\begin{equation}\label{change var27}
\ddot{x}(s) + \dfrac{1+ \ddot{\tau}(s)}{\dot{\tau}(s)}\dot{x}(s) + \nabla f\Big(x(s)+   \dot{\tau}(s) \dot{x}(s) \Big)  = 0 .
\end{equation}
In doing so, we passed from the  first-order  differential equation    \eqref{change var2-b} to the second-order differential equation   \eqref{change var27}, with the advantage that now the coefficient in front of the gradient is fixed.
Let us now particularize the time scale $\tau (\cdot)$.

\paragraph{Open loop control: link with Nesterov method}

According to the Su, Boyd and Cand\`es \cite{SBC} model for the Nesterov method we consider the case where the viscous damping coefficient in \eqref{change var27} satisfies
\begin{equation}\label{change var270}
\dfrac{1+ \ddot{\tau}(s)}{\dot{\tau}(s)} = \dfrac{\alpha}{s} .
\end{equation}
Set $\theta (s)= \dot{\tau}(s)$. We are led to solve the fist-order linear differential equation
\begin{equation}\label{change var271}
\dot{\theta}(s)  - \dfrac{\alpha}{s} \theta (s)=-1.
\end{equation}
After multiplication by $s^{-\alpha}$, we get equivalently
\begin{equation}\label{change var272}
\dfrac{d}{ds}\left( s^{-\alpha} \theta (s)  \right) = -s^{-\alpha}.
\end{equation}
By integrating we get
\begin{equation}\label{change var273}
s^{-\alpha} \theta (s)   = -\dfrac{1}{-\alpha +1} s^{-\alpha +1} + C_0,
\end{equation}
which gives
\begin{equation}\label{change var274}
\theta (s)   = \dfrac{s}{\alpha -1}  + C_0 s^{\alpha} .
\end{equation}
According to $\theta (s)= \dot{\tau}(s)$, we finally get
\begin{equation}\label{change var275}
\tau (s)   = \dfrac{s^2}{2(\alpha -1)}  + {C_1} s^{\alpha +1} + {C_2},
\end{equation}
for ${C_1,C_2}$ real constants.
This leads to the choice $\alpha >1$. A particular and relatively simple situation is to take $C_1=C_2=0$, which gives
\begin{equation}\label{change var275b}
\tau (s)   = \dfrac{s^2}{2(\alpha -1)} .
\end{equation}
Equation \eqref{change var2-b} becomes
\begin{equation}\label{change var2-bb}
\dot{y} (s)      +   \frac{s}{\alpha -1} \nabla f(y(s)) =0,
\end{equation}
and \eqref{SD_rescale_1} gives
\begin{equation}\label{SD_rescale_11}
f (y(s))- \inf_{\cH} f = o \left( \dfrac{1}{s^{2}} \right) \textrm{ as } s \to + \infty.
\end{equation}
From \eqref{change var27} we obtain
\begin{equation}\label{change var27 new}
\ddot{x}(s) + \frac{\alpha}{s}\dot{x}(s) + \nabla f\left(x(s)+   \frac{s}{\alpha -1}\dot{x}(s) \right)  = 0 .
\end{equation}
In doing so, we passed from the  first-order  differential equation    \begin{equation}\label{change var2-bbb}
\dot{y} (s)      +   \frac{s}{\alpha -1} \nabla f(y(s)) =0
\end{equation}
to the second-order differential equation   \eqref{change var27 new}, with the advantage that now the coefficient in front of the gradient is fixed.
Of course, we have to prove that the fast convergence rates are preserved.

The above dynamic \eqref{change var27 new} is a particular case of the Inertial System with Implicit Hessian Damping 
\begin{equation}
{\rm (ISIHD)}\quad\ddot{x}(s)+\frac{\alpha}{s} \dot{x}(s)+\nabla f\Big(x(s)+\beta(s)\dot{x}(s)\Big)=0 ,
\end{equation}
considered by Alecsa, L\'aszl\'o, and Pin\c ta in \cite{ALP}, see also Attouch, Fadili, and Kungurtsev \cite{AFK} in the perturbed case.
The rationale justifying the use of the term ``implicit" comes from the observation that by Taylor expansion (as $s \to +\infty$ we have $\dot{x}(s) \to 0$) one has
\[
\nabla f\Big(x(s)+\beta(s)\dot{x}(s)\Big)\approx \nabla f (x(s)) + \beta(s)\nabla^2 f(x(s))\dot{x}(s) ,
\]
hence making the Hessian damping appear indirectly. 
We can therefore expect the dynamic \eqref{change var27 new} to have an asymptotic behavior similar to that of
\begin{equation}\label{change var277}
\ddot{x}(s) + \frac{\alpha}{s}\dot{x}(s) +  \frac{s}{\alpha -1} \nabla^2 f(x(s))\dot{x}(s)+ \nabla f\left(x(s) \right) 
= 0 .
\end{equation}

\paragraph{Closed loop control}
The continuous steepest descent system (SD) provides several quantities which are strictly increasing and converge to $+ \infty$, so they are eligible for time scaling. 
Since $\| \nabla f(z(t)) \| $ is monotonically decreasing to zero, taking 
\vspace{-10pt}
\begin{equation*}
t = \tau(s) = \dfrac{1}{\| \nabla f(z(s)) \|^p},
\end{equation*}
is eligible for any $p>0$. According to (SD) we have that the same holds true for
$$
t = \tau(s) = \dfrac{1}{\| \dot{z}(s) \|^p}. 
$$
Since $\tau$ comes into the scaling technique via its derivative, we can also consider for $p>0$
$$
t = \dot{\tau}(s) = \dfrac{1}{\| \dot{z}(s) \|^p}. 
$$
Alternatively, we may also consider building $\tau$ with the help of the function values since $f(z(\cdot))$ is decreasing. 

All these choices will lead to closed-loop dynamical systems which have attracted of many researchers recently, see for instance \cite{ABC, LJ}.

\subsection{Convergence rates}\label{subsec:rates}

The above results suggest that similar results are valid for  \eqref{change var27}.
We will examine the convergence rates satisfied by the trajectories of \eqref{change var27} simply by using  time scale and averaging arguments. We do not make any other Lyapunov analysis, we just use the Lyapunov analysis of the steepest descent, which was recalled in Subsection 2.1. Let us state our first result which concerns the implicit Hessian driven damping dynamics.
\begin{theorem}\label{Thm-rescale-averaging-implicit}
{Suppose that  $f \colon \cH \to \R$ satisfies $(\mathcal A)$.}	
Let $x: [s_{0}, +\infty[ \to \cH$ be a solution trajectory of 	
\begin{equation}\label{change var27_thm}
	\ddot{x}(s) + \frac{\alpha}{s}\dot{x}(s) + \nabla f\left(x(s)+   \frac{s}{\alpha -1}\dot{x}(s) \right)  = 0 .
\end{equation}
{Assume that $\alpha > 1$. Then} the following statements are true:
\begin{enumerate}		
	\item \label{rescale-averaging:int-grad} (integral estimate of the gradients) \quad 
	$
	\int_{s_{0}}^{+\infty} s^3 \left\lVert \nabla f \left(x(s) + \frac{s}{\alpha -1}\dot{x}(s)\right) \right\rVert ^{2}  ds < +\infty;
	$
	
	\item \label{rescale-averaging:rate-grad} (convergence  of  gradients towards zero) \quad
	$
	\left\lVert \nabla f \left(x(s) + \frac{s}{\alpha -1}\dot{x}(s)\right) \right\rVert = o \left( \frac{1}{s^{2}} \right) \textrm{ as } s \to + \infty;
	$		
	
	
	\item \label{rescale-averaging:rate-vel} (convergence of velocities towards zero) \quad 
	$
	\left\lVert \dot{x} \left( s \right) \right\rVert = o \left( \frac{1}{s} \right) \textrm{ as } s \to + \infty;
	$						
	
	\item \label{rescale-averaging:traj}
	the solution trajectory $x(s)$ converges weakly as $s \to +\infty$, and its limit belongs to $S=\argmin f$.\vspace{1ex}
	
	If $\alpha>3$, then
	\item \label{rescale-averaging:rate-fun} (convergence of values towards minimal value) \quad
	$
	f(x(s)) -\inf_{\cH} f = o \left( \frac{1}{s^{2}} \right) \textrm{ as } s \to + \infty.
	$	
\end{enumerate}
\end{theorem}
\begin{proof}
First observe that the time scaling and averaging operations allow us to obtain any solution trajectory of \eqref{change var27_thm}. To this end we need to choose adequately the initial data in the first order evolution equations \eqref{change var2-bbb} and \eqref{change var24} which are respectively attached to the scaling and averaging  operations. 
So let us give a solution trajectory $x(\cdot)$ of \eqref{change var27_thm} which satisfies the Cauchy data $x(s_{0}) =x_0$ and $\dot{x}(s_{0})=x_1  $.
Let us verify that this solution is reached by taking  successively	
\begin{equation}\label{eq:change var2-bbbb}
	\begin{cases}
		\dot{y} (s)      +   \frac{s}{\alpha -1} \nabla f(y(s)) =0  \hspace{7mm}
		\vspace{2mm}\\
		y(s_{0})= x_0 + \frac{s_{0}}{\alpha -1}x_1,
	\end{cases}
\end{equation}
then 
\begin{equation}\label{eq:change var24_b}
	\begin{cases}
		\dot{x}(s) + \frac{\alpha -1}{s}(x(s)-y(s))  = 0 
		\vspace{2mm}\\
		x(s_{0})= x_0.
	\end{cases}
\end{equation}
Indeed, \eqref{eq:change var24_b} gives $\dot{x}(s_{0}) = \frac{\alpha -1}{s_{0}}(y(s_{0})-x(s_{0}))$, which, by the second equation of \eqref{eq:change var2-bbbb}, leads to $\dot{x}(s_{0}) =x_1$.

Let us first interpret the passage from $y$ to $x$ as an averaging process.
To this end rewrite \eqref{eq:change var24_b}  as
\begin{equation}\label{change var28}
	s \dot{x}(s) + (\alpha -1) x(s) = (\alpha -1) y(s).
\end{equation}

\noindent After multiplication of  \eqref{change var28} by $s^{\alpha -2}$, we get equivalently
\begin{equation}\label{change var29}
	s^{\alpha -1} \dot{x}(s) + (\alpha -1)s^{\alpha -2} x(s) = (\alpha -1)s^{\alpha -2} y(s),
\end{equation}
that is
\begin{equation}\label{change var30}
	\frac{d}{ds} \left( s^{\alpha -1}x(s)\right)  = (\alpha -1)s^{\alpha -2} y(s).
\end{equation}
By integrating \eqref{change var30} from $s_{0}$ to $s$, and according to  $ x(s_{0})=x_0$, we obtain
\begin{eqnarray}
	x(s) &=&  \frac{s_{0}^{\alpha -1}}{ s^{\alpha -1}} x_0 + \frac{\alpha -1}{s^{\alpha -1}}\int_{s_{0}}^s u^{\alpha -2} y(u)du \label{def:x}\\
	&=&  \frac{s_{0}^{\alpha -1}}{ s^{\alpha -1}} y(s_{0}) + \frac{\alpha -1}{s^{\alpha -1}}\int_{s_{0}}^s u^{\alpha -2} y(u)du - \frac{{s_{0}}^{\alpha}}{(\alpha -1)s^{\alpha -1}}x_1 
\end{eqnarray}
where the last equality follows from the choice of $y(s_{0})$ as given by
\eqref{eq:change var2-bbbb}.  
Then, observe that $x(s)$ can be simply written as follows
\begin{equation}\label{proba-formulation}
	x(s) =   \int_{s_{0}}^s y(u)\,  d\mu_{s} (u) + \xi(s),
\end{equation}
where $\mu_s$ is the positive  Radon  measure on $[s_{0}, s]$ defined by
$$
\mu_s = \frac{s_{0}^{\alpha -1}}{ s^{\alpha -1}} \delta_{s_{0}} +  (\alpha -1) \frac{u^{\alpha -2}}{s^{\alpha -1}} du,
$$
where $\delta_{s_{0}}$ is the Dirac measure at $s_{0}$, and
$$
\xi(s):=- \frac{{s_{0}}^{\alpha}}{(\alpha -1)s^{\alpha -1}} x_1
$$
is considered as a small perturbation for large values of $s$.

\ref{rescale-averaging:int-grad} $\&$ \ref{rescale-averaging:rate-grad}
According to the  energy estimates for the continuous steepest descent system, see Theorem \ref{theorem SD pert},  we have that for any solution trajectory of
\begin{equation}\label{SD_b}
	{\rm (SD)} \quad   \dot{z}(t) + \nabla f (z(t)) =0
\end{equation}
it holds
\begin{equation}\label{SD_value_b}
	{\int_{t_0}^{+\infty} t \left\lVert \dot{z} \left( t \right) \right\rVert ^{2} dt < + \infty, \quad} \int_{t_0}^{+\infty} t \| \nabla f (z(t)  \|^2 dt  < +\infty, {\quad \textrm{ and } \quad \left\lVert \nabla f \left( z \left( t \right) \right) \right\rVert = o \left( \frac{1}{t} \right) \textrm{ as } t \to + \infty} .
\end{equation}
Notice that from the unperturbed (SD) we also have
\begin{equation}\label{SD_value_f}
	\left\lVert \dot{z} \left( t \right) \right\rVert = o \left( \frac{1}{t} \right) \textrm{ as } t \to + \infty .
\end{equation}

After making the change of time variable 
$$t=u (s)= \frac{s^2}{2(\alpha -1)},$$
and setting $z(u (s))= y(s)$, 
the {second integral estimate  above} becomes
\begin{equation}\label{SD_value_c}
	\int_{s_{0}}^{+\infty} u (s) \| \nabla f (z(u (s))  \|^2 \dot{u}(s) ds < +\infty,
\end{equation}
that is 
\begin{equation}\label{SD_value_d}
	\int_{s_{0}}^{+\infty} s^3 \| \nabla f (y (s))  \|^2  ds < +\infty.
\end{equation}
Replacing $y$ by {its equivalent formulation $y(s)= x(s)+   \frac{s}{\alpha -1}\dot{x}(s)$ in \eqref{SD_value_b} and} \eqref{SD_value_d} gives the claims.

\ref{rescale-averaging:rate-vel}
{From \eqref{SD_value_f} we have that  there exists a positive function $\varepsilon_0 (\cdot)$ which satisfies
	$\lim_{s \to +\infty} \varepsilon_0 (s) =0  $ such that}
\begin{equation}\label{small_o_20}
	\| \dot{z} (\tau(s)) \|= \frac{\varepsilon_0 (\tau(s))}{\tau(s)} ,
\end{equation}
and since $ \dot{y}(s) = \dot{\tau} (s)\dot{z} (\tau(s)) $
we get

\begin{equation}\label{small_o_21}
	\| \dot{y} (s) \|= \frac{\varepsilon_0 (\tau (s)) \dot{\tau} (s)}{\tau(s)} .
\end{equation}
From $\tau (s)   = \frac{s^2}{2(\alpha -1)}$, we get

\begin{equation}\label{small_o_22}
	\| \dot{y} (s) \|= \frac{\varepsilon_{1} (s) }{s} ,
\end{equation}
with $\varepsilon_{1} (s) = 2\varepsilon_0 (\tau (s)) $ which tends to zero
as  $s \to +\infty$. Let us now establish a similar estimate for $\| \dot{x} (s) \| $.
According to \eqref{def:x}
\begin{eqnarray}
	x(s) &=&  \frac{s_{0}^{\alpha -1}}{ s^{\alpha -1}} x_0 + \frac{\alpha -1}{s^{\alpha -1}}\int_{s_{0}}^s u^{\alpha -2} y(u)du.
\end{eqnarray}
Let us derivate this expression. We get
\begin{eqnarray}\label{convolution1}
	\dot{x}(s) &=&  -\frac{(\alpha -1)s_{0}^{\alpha -1}}{ s^{\alpha}} x_0 - \frac{(\alpha -1)^2}{s^{\alpha}}\int_{s_{0}}^s u^{\alpha -2} y(u)du  + (\alpha -1) \frac{1}{s} y(s) .
\end{eqnarray}
Let us reformulate this expression in terms of $\dot{y}(s)$, which is the quantity whose speed of convergence is known by  \eqref{small_o_22}.
Indeed, by integration by parts we have
\begin{eqnarray*}
	\int_{s_{0}}^s  u^{\alpha -1} \dot{y}(u)du &=& s^{\alpha -1} y(s)- 
	s_0^{\alpha -1} y(s_0) -(\alpha -1) \int_{s_{0}}^s u^{\alpha -2} y(u)du   .
\end{eqnarray*}
After multiplication of the above expression by $\frac{\alpha -1}{s^{\alpha}} $, and according to $y(s_0)=x_0 {+ \frac{s_0}{\alpha-1} x_{1}}$,  we get
\begin{eqnarray}\label{convolution2}
	\frac{\alpha -1}{s^{\alpha}} \int_{s_{0}}^s  u^{\alpha -1} \dot{y}(u)du &=& (\alpha -1) \frac{1}{s} y(s) -\frac{(\alpha -1)s_{0}^{\alpha -1}}{ s^{\alpha}} x_0
	{- \frac{s_{0}^{\alpha}}{s^{\alpha}} x_{1}}
	-\frac{(\alpha -1)^2}{s^{\alpha}}\int_{s_{0}}^s u^{\alpha -2} y(u)du  .
\end{eqnarray}
Comparing \eqref{convolution1} and \eqref{convolution2} we get
\begin{eqnarray*}
	\dot{x}(s) &=& \frac{\alpha -1}{s^{\alpha}} \int_{s_{0}}^s  u^{\alpha -1} \dot{y}(u)du {+ \frac{s_{0}^{\alpha}}{s^{\alpha}} x_{1}} .
\end{eqnarray*}
According to \eqref{small_o_22}, we obtain successively
\begin{eqnarray}
	\| \dot{x}(s)\|	&\leq &  \frac{\alpha -1}{s^{\alpha }} \int_{s_{0}}^s  u^{\alpha -1} \left\lVert \dot{y}(u) \right\rVert du  {+ \frac{s_{0}^{\alpha}}{s^{\alpha}} \left\lVert x_{1} \right\rVert}  \nonumber \\
	&= & \frac{\alpha -1}{s^{\alpha }} \int_{s_{0}}^s  u^{\alpha -2}
	\varepsilon_{1} (u) du  {+ \frac{s_{0}^{\alpha}}{s^{\alpha}} \left\lVert x_{1} \right\rVert} \label{convolution3} .
\end{eqnarray}
After multiplication of \eqref{convolution3}    by $s$, we obtain
\begin{eqnarray*}
	s\| \dot{x}(s)\|	&\leq &  \frac{\alpha -1}{s^{\alpha -1}} \int_{s_{0}}^s  u^{\alpha -2} 
	\varepsilon_{1} (u) du + \frac{s_{0}^{\alpha}}{s^{\alpha-1}} \left\lVert x_{1} \right\rVert,
\end{eqnarray*}
and Lemma \ref{lem lim-0} gives us $\lim_{s \to +\infty}\| s\dot{x}(s) \|=0$. 

\ref{rescale-averaging:traj}
For trajectory convergence, following the approach of the paper, we do not perform a Lyapunov analysis, but take advantage of the fact that the trajectory $z \left( s \right)$ of the steepest descent dynamics converges weakly towards a solution $x_{*} \in S$ as $s \to +\infty$. This immediately implies that $y(s) = z(u(s))$ converges weakly to $x_{*}$ as $s \to +\infty$. In other words, for each $v \in \cH$	
$$
\left\langle y \left( s \right) , v \right\rangle \to \left\langle x_{*} , v \right\rangle \textrm{ as } s \to + \infty .
$$
To pass from the convergence of $y$ to that of $x$, we use the interpretation of $x$ as an average of $y$ plus a negligeable term. The convergence then results from the general property which says that convergence entails ergodic convergence. Let us make this precise.
Since the perturbation term $\xi$ is negligible as $s\to +\infty$,
we have by definition of $x(s)$
\begin{equation*}
	x(s) \sim  \frac{s_{0}^{\alpha -1}}{ s^{\alpha -1}} y(s_{0})+ \frac{\alpha -1}{s^{\alpha -1}}\int_{s_{0}}^s u^{\alpha -2} y(u)du.
\end{equation*}
After elementary calculus, we just need to prove that for $a(\cdot)$ be a positive real valued function which verifies $\lim_{s \to + \infty} a(s)=0$, then $\lim_{s \to + \infty} A(s)=0$, where
$$
A(s) =   \frac{1}{s^{\alpha -1}}\int_{s_{0}}^s u^{\alpha -2} a(u)du,
$$
{which is true according to Lemma \ref{lem lim-0}.}

\ref{rescale-averaging:rate-fun} First, let us get rid of the perturbation term. 
{From \ref{rescale-averaging:traj}} we deduce that $x(\cdot)$ is also bounded.
Since $\nabla f$ is Lipschitz continuous on the bounded sets, it follows that $\nabla f$ is bounded on the bounded sets, which immediately implies that $f$ is also Lipschitz continuous on the bounded sets. Therefore there exists $L_f >0$ such that
\begin{equation*}
	f(x(s))- f\left(\int_{s_{0}}^s y(u)\,  d\mu_{s} (u)\right) \leq L_f \|\xi(s)\|\leq \frac{{C_f}}{s^{\alpha -1}} \quad \forall s \geq s_0,
\end{equation*}
for $C_{f} := \frac{L_f {s_{0}}^{\alpha}}{\alpha -1} \left\lVert x_1 \right\rVert$, which leads to
\begin{equation}\label{eq:18:51}
	f(x(s)) -\inf_{\cH} f \leq 
	f\left(\int_{s_{0}}^s y(u)\,  d\mu_{s} (u)\right) -\inf_{\cH} f + \frac{C_f}{s^{\alpha -1}} \quad \forall s \geq s_0.  
\end{equation}

We are therefore reduced to examining the convergence rate of  $f\left(\int_{s_{0}}^s y(u)  d\mu_{s} (u)\right) -\inf f$ towards zero as $s \to + \infty$.
We have  that $\mu_s$ is a positive Radon measure on $[s_{0}, s]$  whose total mass is equal to $1$. It is therefore a probability measure, and $\int_{s_{0}}^s y(u)\,  d\mu_{s} (u)$ is obtained by \textbf{averaging} the trajectory $y(\cdot)$ on $[s_{0},s]$ with respect to  $\mu_s$.
From there, we can deduce fast convergence properties for the solution trajectories of \eqref{change var27_thm}.
According to the convexity of $f$, and  Jensen's inequality, we obtain that
\begin{equation}
	\label{eq:Jensen}
	f\left(\int_{s_{0}}^s y(u)\,  d\mu_{s} (u)\right)  -\inf_{\cH} f 
	\leq \int_{s_{0}}^s  \left( f (y(u)) -\inf_{\cH} f \right) d\mu_s (u) .
\end{equation}

It follows from \eqref{SD_rescale_11} that there exists a positive function $\varepsilon (\cdot)$ which satisfies
$\lim_{s \to +\infty} \varepsilon (s) =0  $ such that
\begin{equation}\label{small_o_1}
	f \left( y \left( s \right) \right) -\inf_{\cH} f = \frac{\varepsilon (s)}{s^{2}} .
\end{equation} 
According to the definition of $\mu_s$, it yields for all $s \geq s_0$
\begin{equation*}
	f\left( \int_{s_{0}}^s y(u)\,  d\mu_{s} (u)\right)  -\inf_{\cH} f 
	\leq  \int_{s_{0}}^s  \frac{\varepsilon (u)}{u^{2}} d\mu_s (u)
	= \frac{s_{0}^{\alpha -3}}{s^{\alpha -1}} + \dfrac{\alpha - 1}{s^{\alpha-1}} \int_{s_{0}}^s \varepsilon (u)u^{\alpha-4} du .
\end{equation*}
By making use of \eqref{eq:18:51}, we deduce
\begin{equation*}
	s^{2} \left( f(x(s)) -\inf_{\cH} f \right)
	\leq  \frac{C_{f} + s_{0}^{\alpha -3}}{s^{\alpha-3}} + \dfrac{\alpha - 1}{s^{\alpha-3}} \int_{s_{0}}^s \varepsilon (u)u^{\alpha-4} du .
\end{equation*}

Therefore, for $\alpha >3$ we get
\begin{eqnarray*}
	\limsup_{s \to +\infty}	s^2 (f\left(x(s)\right)  -\inf_{\cH} f) 
	&\leq & \left( \alpha - 1 \right) \limsup_{s \to +\infty}
	\dfrac{1}{s^{\alpha-3}}	 \int_{s_{0}}^s \varepsilon (u) u^{\alpha -4}du.
\end{eqnarray*}
{Finally, we just need to apply Lemma \ref{lem lim-0} to prove the claim.}
\end{proof}

\begin{remark}\label{remark alpha>1}
Let us emphasize the fact that the convergence of the solution trajectories of the damped inertial system \eqref{change var27_thm}
is valid under the condition $\alpha >1$.
This  contrasts with the stronger condition $\alpha >3$ which is required to obtain the convergence of the trajectories for the  system without Hessian driven damping	
\begin{equation}\label{change var27_thm_rem_b}
	\ddot{x}(s) + \frac{\alpha}{s}\dot{x}(s) + \nabla f\left(x(s) \right)  = 0 ,
\end{equation}
and which is the low-resolution ODE of  Nesterov's accelerated gradient method. This property is  related to the presence of the Hessian driven damping and the particular form of the corresponding coefficient  $\frac{s}{\alpha -1}.$ 	
Moreover, it has been established in \cite{ACPR, AP} that the choice $\alpha >3$ provides a convergence rate of the values of $o \left( 1/s^2 \right)$ as $s \rightarrow +\infty$,  and therefore it improves the convergence rate of Nesterov's accelrated gradient method of $\cO \left( 1 /s^2 \right) $ as $s \rightarrow +\infty$.
As seen, the time scale and averaging approach is flexible enough to also provide the convergence rate $o \left( 1 /s^2 \right)$ as $s \rightarrow +\infty$,  in the case $\alpha >3$.
\end{remark}

\begin{remark}
The time scale and averaging technique provides a new approach which allows to transfer the convergence rate of the function value along the trajectory $y(s)$ obtained via the time scaling of (SD) to the function value along the trajectory $x(s)$ that can be seen as the average of the initial one.  This phenomenon can be seen in analogy with the manner the Nesterov method and the Ravine method behave to each other. 

Recall that the Nesterov'acceleration gradient \cite{Nest1,Nest2} reads
\begin{equation}\label{Nesterov:scheme}
	(\forall k \geq 1) \quad \begin{cases}
		y_{k} & := x_{k} + \alpha_{k} \left( x_{k} - x_{k-1} \right)
		\vspace{1mm}\\
		x_{k+1}	& := y_{k} - \lambda \nabla f \left( y_{k} \right)
	\end{cases},
\end{equation}
where $x_{0} = x_{1} \in \cH$, $0 < \lambda \leq \frac{1}{L}$ and
\begin{equation*}
	\alpha_{k} := \dfrac{t_{k}-1}{t_{k+1}} \quad \textrm{ with } \quad \begin{cases}
		t_{1} := 1 \\
		t_{k+1}^{2} - t_{k+1} \leq t_{k}^{2} \quad \forall k \geq 1
	\end{cases} .
\end{equation*}

The so-called Ravine sequence $(y_k)_{k \geq 1}$ is the one generated by the Ravine method introduced by Gelfand and Tsetlin \cite{GT,AF}.  We will show that the values of $f$ along the Ravine sequence converge fast towards $\inf_{\cH}$ and that the convergence rate can be transferred to the values of $f$ along  $(x_k)_{k \geq 0}$,  which is an average sequence of $(y_k)_{k \geq 1}$. Indeed,  by induction arguments,  we have for every $k \geq 1$
\begin{align*}
	x_{k+1} = \dfrac{1}{1 + \alpha_{k+1}} y_{k+1} + \dfrac{\alpha_{k+1}}{1 + \alpha_{k+1}} x_{k} & = \dfrac{1}{1 + \alpha_{k+1}} y_{k+1} + \dfrac{\alpha_{k+1}}{1 + \alpha_{k+1}} \left( \dfrac{1}{1 + \alpha_{k}} y_{k} + \dfrac{\alpha_{k}}{1 + \alpha_{k}} x_{k-1} \right) \nonumber \\
	& = \cdots = \sum_{i=1}^{k+1} \theta_{k+1,i} y_{i} ,
\end{align*}
where the nonnegative weights $\left( \theta_{k+1,i} \right) _{1 \leq i \leq k+1}$ defined by
\begin{equation*}
	\theta_{k+1,i} := \dfrac{1}{1 + \alpha_{k+1}} \prod_{j=1}^{k+1-i} \dfrac{\alpha_{k+2-j}}{1 + \alpha_{k+1-j}} \quad \forall 1 \leq i \leq k \quad \mbox{and} \ \theta_{k+1,k+1}:=  \dfrac{1}{1 + \alpha_{k+1}}
\end{equation*}
fulfill $\sum_{i=1}^{k+1} \theta_{k+1,i} = 1$.

In the following we give a direct proof for the fast convergence rate of the Ravine method, which can be seen as an alternative approach to \cite{AF}.
Let $z_{*} \in S$. We define for every $k \geq 2$ the discrete energy function
\begin{equation*}
	E_{k} := t_{k}^{2} \left( f \left( y_{k-1} \right) - \inf_{\cH} f - \lambda \left( 1 - \dfrac{L \lambda}{2} \right) \left\lVert \nabla f \left( y_{k-1} \right) \right\rVert ^{2} \right) + \dfrac{1}{2 \lambda} \left\lVert \left( t_{k} - 1 \right) \left( x_{k} - x_{k-1} \right) + x_{k} - z_{*} \right\rVert ^{2} \geq 0 .
\end{equation*}
Let $k \geq 2$. By the convexity of $f$, it yields 
\begin{align*}
	f \left( x_{k} \right) \geq f \left( y_{k} \right) + \left\langle \nabla f \left( y_{k} \right) , x_{k} - y_{k} \right\rangle  \quad \mbox{and} \quad \inf_{\cH} f \geq f \left( y_{k} \right) + \left\langle \nabla f \left( y_{k} \right) , z_{*} - y_{k} \right\rangle .
\end{align*}
By multiplying the first inequality by $t_{k+1}^{2} - t_{k+1} > 0$ and the second one by $t_{k+1} > 0$ and summing the resulting inequalities, it yields
\begin{align}
	\left( t_{k+1}^{2} - t_{k+1} \right) \left( f \left( x_{k} \right) - \inf_{\cH} f \right) & \geq t_{k+1}^{2} \left( f \left( y_{k} \right) - \inf_{\cH} f \right) - t_{k+1} \left\langle \nabla f \left( y_{k} \right) , \left( t_{k+1} - 1 \right) \left( y_{k} - x_{k} \right) + y_{k} - z_{*} \right\rangle. \label{Ravine:mul-sum}
\end{align}
Since $t_{k+1} y_{k} = t_{k+1} x_{k} + \left( t_{k} - 1 \right) \left( x_{k} - x_{k-1} \right)$,  by denoting
\begin{align*}
	u_{k} := \left( t_{k+1} - 1 \right) \left( y_{k} - x_{k} \right) + y_{k} - z_{*} = t_{k+1} \left( y_{k} - x_{k} \right) + x_{k} - z_{*} = \left( t_{k} - 1 \right) \left( x_{k} - x_{k-1} \right) + x_{k} - z_{*} ,
\end{align*}
we have
\begin{align*}
	- \lambda t_{k+1} \nabla f \left( y_{k} \right) = t_{k+1} \left( x_{k+1} - y_{k} \right) = t_{k+1} \left( x_{k+1} - x_{k} \right) - \left( t_{k} - 1 \right) \left( x_{k} - x_{k-1} \right) = u_{k+1} - u_{k} .
\end{align*}
Therefore
\begin{align*}
	& \ - t_{k+1} \left\langle \nabla f \left( y_{k} \right) , \left( t_{k+1} - 1 \right) \left( y_{k} - x_{k} \right) + y_{k} - z_{*} \right\rangle \nonumber \\
	=  & \ \dfrac{1}{\lambda} \left\langle u_{k+1} - u_{k} , u_{k} \right\rangle = \dfrac{1}{2 \lambda} \left( \left\lVert u_{k+1} \right\rVert ^{2} - \left\lVert u_{k+1} - u_{k} \right\rVert ^{2} - \left\lVert u_{k} \right\rVert ^{2} \right) \nonumber \\
	= & \ \dfrac{1}{2 \lambda} \left( \left\lVert \left( t_{k+1} - 1 \right) \left( x_{k+1} - x_{k} \right) + x_{k+1} - z_{*} \right\rVert ^{2} - \left\lVert \left( t_{k} - 1 \right) \left( x_{k} - x_{k-1} \right) + x_{k} - z_{*} \right\rVert ^{2} \right) - \dfrac{1}{2} \lambda t_{k+1}^{2} \left\lVert \nabla f \left( y_{k} \right) \right\rVert ^{2} .
\end{align*}
Replacing this expression into \eqref{Ravine:mul-sum},  after some rearrangements we obtain
\begin{align*}
	E_{k+1} \leq \dfrac{1}{2 \lambda} \left\lVert \left( t_{k} - 1 \right) \left( x_{k} - x_{k-1} \right) + x_{k} - z_{*} \right\rVert ^{2} + \left( t_{k+1}^{2} - t_{k+1} \right) \left( f \left( x_{k} \right) - \inf_{\cH} f \right) - \dfrac{1}{2} \lambda \left( 1 - L \lambda \right) t_{k+1}^{2} \left\lVert \nabla f \left( y_{k} \right) \right\rVert ^{2} .
\end{align*}
In addition,  the Descent Lemma gives
\begin{align}
	f \left( x_{k} \right) - \inf_{\cH} f & \leq f \left( y_{k-1} \right) - \inf_{\cH} f + \left\langle \nabla f \left( y_{k-1} \right) , x_{k} - y_{k-1} \right\rangle + \dfrac{L}{2} \left\lVert x_{k} - y_{k-1} \right\rVert ^{2} \nonumber \\
	& = f \left( y_{k-1} \right) - \inf_{\cH} f - \lambda \left( 1 - \dfrac{L \lambda}{2} \right) \left\lVert \nabla f \left( y_{k-1} \right) \right\rVert ^{2} , \label{Ravine:Descent}
\end{align}
which further leads to the inequality
\begin{align*}
	E_{k+1} \leq & \  E_{k} - \dfrac{1}{2} \lambda \left( 1 - L \lambda \right) t_{k+1}^{2} \left\lVert \nabla f \left( y_{k} \right) \right\rVert ^{2} \nonumber \\
	& \ + \left( t_{k+1}^{2} - t_{k+1} - t_{k}^{2} \right) \left( f \left( y_{k-1} \right) - \inf_{\cH} f - \lambda \left( 1 - \dfrac{L \lambda}{2} \right) \left\lVert \nabla f \left( y_{k-1} \right) \right\rVert ^{2} \right) ,
\end{align*}
and it also proves that $E_{k} \geq 0$ for every $k \geq 2$.
Therefore, when $0 < \lambda < \frac{1}{L}$, we can deduce
\begin{equation*}
	f \left( y_{k} \right) - \inf_{\cH} f = \mathcal{O} \left( \dfrac{1}{t_{k+1}^{2}} \right) = \mathcal{O} \left( \dfrac{1}{(k+1)^{2}} \right) ,
\end{equation*}
which immediately transfers to
\begin{equation*}
	f \left( x_{k} \right) - \inf_{\cH} f = \mathcal{O} \left( \dfrac{1}{t_{k+1}^{2}} \right) = \mathcal{O} \left( \dfrac{1}{(k+1)^{2}} \right) .
\end{equation*}
\end{remark}

\section{Applications to other continuous dynamics} {In this section we will show that the proposed approach which combines time scaling with averaging can be successfully applied beyond the classical steepest descent dynamical system.}

\subsection{Explicit Hessian damping}\label{sec:explicit}

We consider an explicit Hessian driven damping version of the Su-Boyd-Cand\'es dynamical system \eqref{SBC}. As explained in \eqref{change var277} the following dynamic
\begin{equation}\label{eq:dyn_inert_implicit}
\ddot{y} \left( s \right) + \dfrac{\alpha}{s} \dot{y} \left( s \right) + \dfrac{s}{\alpha +1} \dfrac{d}{ds} \left( \nabla f \left( y \left( s \right) \right) \right) + \nabla f \left( y \left( s \right) \right) = 0 
\end{equation}
comes naturally by performing a Taylor expansion in \eqref{change var27_thm}. 
Just note that to match the general writing of inertial dynamics for optimization, the above formula is obtained by taking $\alpha +1$ instead of $\alpha -1$.


The above dynamic \eqref{eq:dyn_inert_implicit} is a particular case of the general dynamic on $[s_0, +\infty[$
$${\rm\mbox{(DIN-AVD)}}_{\alpha, \beta, b} \quad  \ddot{y}(s) + \displaystyle{\frac{\alpha}{s} }\dot{y}(s) +   \beta(s) \nabla^2  f (y(s)) \dot{y} (s) + b(s) \nabla  f (y(s)) = 0,
$$
studied by Attouch, Chbani, Fadili, Riahi in \cite{ACFR,ACFR-Opti}. 
Let us recall the convergence results obtained in \cite{ACFR,ACFR-Opti} by performing a Lyapunov analysis for ${\rm (DIN-AVD)}_{\alpha, \beta, b} $.
They will  serve as a comparison with our results.
\begin{theorem} \label{ACFR,rescale} Assume that $\alpha \geq 1$. \; Set 
$w(s)\eqdef b(s)-\left( \dot{\beta}(s) +\dfrac{\beta(s)}{s}\right)$
and suppose that the following conditions are satisfied for every $s \geq s_0$
\begin{eqnarray*}
	&& (\mathcal{G}_{2}) \quad b(s) > \dot{\beta}(s) +\dfrac{\beta(s)}{s}; 
	\\
	&& (\mathcal{G}_{3}) \quad  s\dot{w}(s)\leq (\alpha-3)w(s).\hspace{3cm}
\end{eqnarray*}

\noindent Then, for every solution trajectory $y: [s_{0}, +\infty[ \rightarrow \cH$ of  ${\rm (DIN-AVD)}_{\alpha, \beta, b} $, the following statements are true:
\begin{eqnarray*}
	&& i) \, \,
	f(y(s))-\inf_{\cH} f = \cO 
	\left(\dfrac{1}{s^2 w(s)}\right) \, \mbox{ as } \, s \to + \infty;
	\\
	&& ii) \, \int_{s_{0}}^{+\infty} s^2 \beta(s) w(s)\|\nabla f(y(s))\|^{2} ds<+\infty;
	\\
	&& iii) \,  \int_{s_{0}}^{+\infty} s \Big(  (\alpha-3)w(s) - s\dot{w}(s)\Big)(f(y(s))-\inf_{\cH} f) ds <+\infty .
\end{eqnarray*}
\end{theorem}

Let us specialize the above result to \eqref{change var27_thm} by taking $b(s) = 1$ and $\beta (s)=\frac{s}{\alpha -1} $ for every $s \geq s_0$. We get $w(\cdot) \equiv \frac{\alpha -3}{\alpha -1}$.
Thus $(\mathcal{G}_{2})$ is satisfied for 
$\alpha >3$, {while}  $(\mathcal{G}_{3})$   reduces to 
$0 \leq \frac{{(\alpha -3)}^2}{\alpha -1}$, which is also satisfied.
We thus get the convergence rates for the trajectories of \eqref{change var27_thm}
$$
f(y(s))-\min_{\cH} f = \cO 
\left(\dfrac{1}{s^2}\right) \, \mbox{ as } \, s \to + \infty;
\, \int_{s_{0}}^{+\infty} s^3\|\nabla f(y(s))\|^{2} ds<+\infty; 
\,  \int_{s_{0}}^{+\infty} s (f(y(s))-\inf_{\cH} f) ds <+\infty .
$$
In this particular setting we can deduce further (see \cite[Theorem 2.1]{ACFR-Opti})
$$
\int_{s_{0}}^{+\infty} s\|\dot{y}(s)\|^{2} ds<+\infty; \quad 
f(y(s))-\min_{\cH} f = o 
\left(\dfrac{1}{s^2}\right) \, \mbox{ and }
\|\dot{y}(s)\| = o
\left(\dfrac{1}{s}\right) \, \mbox{ as } \, s \to + \infty.
$$

Let us show that this dynamic can be reached by time scaling of a perturbed version of the continuous steepest descent dynamical system. As for the implicit case, we make in the latter the change of time variable 
$$\tau (s)= \frac{s^2}{2(\alpha +1)}.$$
Thus we obtain the rescaled steepest descent  
\begin{equation}\label{change var2-bbbb}
\dot{y} (s)      +   \frac{s}{\alpha +1} \nabla f(y(s)) =0,
\end{equation}
with the convergence rate
\begin{equation}\label{SD_rescale_111}
f (y(s))- \inf_{\cH} f {= o \left(\dfrac{1}{s^2}\right) \mbox{ as } \, s \to + \infty.}
\end{equation}
To obtain an explicit Hessian driven damping term, we take the derivative with respect to $s$ in \eqref{change var2-bbbb}
$$
\ddot{y} \left( s \right) + \dfrac{1}{\alpha + 1} \nabla f \left( y \left( s \right) \right) + \dfrac{s}{\alpha + 1} \dfrac{d}{ds} \left( \nabla f \left( y \left( s \right) \right) \right) = 0 .
$$
On the other hand, for every $s \geq \s_{0}$, by multiplying both sides of \eqref{change var2-bbbb} by $\frac{\alpha}{s} > 0$, we get 
$$
\dfrac{\alpha}{s} \dot{y} \left( s \right) + \dfrac{\alpha}{\alpha + 1} \nabla f \left( y \left( s \right) \right) = 0 .
$$
Summing up the two above equations above yields
$$
\ddot{y} \left( s \right) + \dfrac{\alpha}{s} \dot{y} \left( s \right) + \dfrac{s}{\alpha +1} \dfrac{d}{ds} \left( \nabla f \left( y \left( s \right) \right) \right) + \nabla f \left( y \left( s \right) \right) = 0 ,
$$
which is precisely \eqref{eq:dyn_inert_implicit}.

\begin{theorem}\label{Thm-rescale-averaging}
{Suppose that  $f \colon \cH \to \R$ satisfies $(\mathcal A)$.}	
Let $y: [s_{0}, +\infty[ \to \cH$ be a solution trajectory of
\begin{equation}\label{Hessian_explicit}
	\ddot{y} \left( s \right) + \dfrac{\alpha}{s} \dot{y} \left( s \right) + \dfrac{s}{\alpha + 1} \dfrac{d}{ds} \left( \nabla f \left( y \left( s \right) \right) \right) + \nabla f \left( y \left( s \right) \right) = 0 .
\end{equation}
\noindent
Assume that $\alpha > 1$. Then the following statements are true:
\begin{enumerate}
	\item \label{rescale-pert:int-vel} 
	(integral estimate of the velocities)
	$
	\int_{s_0}^{+\infty} s \left\lVert \dot{y} \left( s \right) \right\rVert ^{2} ds < + \infty;
	$
	
	\item \label{rescale-pert:int-fun} 
	(integral estimate of the values)
	$
	\int_{s_0}^{+\infty} s \left( f \left( y \left( s \right) \right) -\inf_{\cH} f \right) ds < +\infty;
	$
	\item \label{rescale-pert:int-grad} 
	(integral estimate of the gradients)
	$
	\int_{s_0}^{+\infty} s^{3} \left\lVert \nabla f \left( y \left( s \right) \right) \right\rVert ^{2}  ds < +\infty;
	$	
	
	\item \label{rescale-pert:rate-fun} 
	(convergence of values towards minimal value)
	$
	f \left( y \left( s \right) \right) -\inf_{\cH} f = o \left( \frac{1}{s^{2}} \right) \textrm{ as } s \to + \infty;
	$
	
	\item \label{rescale-pert:traj}
	the solution trajectory $y(s)$ converges weakly as $s \to +\infty$, and its limit belongs to $S=\argmin f$.\vspace{1ex}
	
	If $\alpha > 2$, then
	\item \label{rescale-pert:rate-vel}
	(convergence of gradients towards zero)	\quad 
	$
	\left\lVert \dot{y} \left( s \right) \right\rVert = o \left( \frac{1}{s} \right) \textrm{ as } s \to + \infty;
	$
	
	\item \label{rescale-pert:rate-grad}
	(convergence of gradients towards zero)	\quad 
	$
	\left\lVert \nabla f \left( y \left( s \right) \right) \right\rVert = o \left( \frac{1}{s^{2}} \right) \textrm{ as } s \to + \infty.
	$
\end{enumerate}
\end{theorem}

\begin{proof}
\noindent
The proof consists in reversing the process described above which has permitted us to pass from (SD) to the damped inertial dynamic with explicit Hessian driven damping \eqref{Hessian_explicit}. Let $y$ be an arbitrary solution trajectory of \eqref{Hessian_explicit} which satisfies the Cauchy data $y \left(s_{0} \right) := y_{0}$ and $\dot{y} \left(s_{0} \right) := y_{1}$. 
Observe that 
\begin{align*}
	\dfrac{d}{ds} \left( s^{\alpha} \dot{y} \left( s \right) \right) & = s^{\alpha} \ddot{y} \left( s \right) + \alpha  s^{\alpha-1} \dot{y} \left( s \right) \\
	\dfrac{d}{ds} \left( \dfrac{s^{\alpha+1}}{\alpha+1} \nabla f \left( y \left( s \right) \right) \right) & = \dfrac{s^{\alpha+1}}{\alpha+1} \dfrac{d}{ds} \nabla f \left( y \left( s \right) \right) + s^{\alpha} \nabla f \left( y \left( s \right) \right) .
\end{align*}
Hence, by multiplying both sides of \eqref{Hessian_explicit} by $s^{\alpha}>0$, we deduce that
\begin{equation*}
	\dfrac{d}{ds} \left( s^{\alpha} \dot{y} \left( s \right) + \dfrac{s^{\alpha+1}}{\alpha+1} \nabla f \left( y \left( s \right) \right) \right) = 0 .
\end{equation*}
This gives, equivalently
\begin{equation}\label{Hessian 02}
	\dot{y} \left( s \right) + \dfrac{s}{\alpha+1} \nabla f \left( y \left( s \right) \right) = \dfrac{c_0}{s^{\alpha}} ,
\end{equation}
where $c_{0} \in \cH$ is a constant that can be determined from the Cauchy condition. Precisely, taking $s := s_{0}$ in \eqref{Hessian 02} it yields 
\begin{equation}\label{Hessian init}
	c_{0} = s_{0}^{\alpha} y_{1} + \frac{s_{0}^{\alpha+1}}{\alpha+1} \nabla f \left( y_{0} \right).
\end{equation}
Next we show that time scaling makes it possible to link the solution trajectory $y(\cdot)$ of \eqref{Hessian_explicit} to the solution trajectory of a perturbation of (SD). 
Set $z \left( t \right) = y \left( \sigma \left( t \right) \right)$, where
\begin{equation*}
	\sigma \left( t \right) = \sqrt{2 \left( \alpha+1 \right) t} .
\end{equation*}
Therefore,
\begin{equation*}
	\dot{\sigma} \left( t \right) = \sqrt{\dfrac{\alpha+1}{2t}} \quad \textrm{ and } \quad \dot{z} \left( t \right) = \dot{\sigma} \left( t \right) \dot{y} \left( \sigma \left( t \right) \right) .
\end{equation*}
Moreover, notice that \eqref{Hessian 02} can be equivalently written as
\begin{equation*}
	\dfrac{\alpha+1}{s} \dot{y} \left( s \right) + \nabla f \left( y \left( s \right) \right) = \dfrac{c_0 (\alpha+1)}{s^{\alpha+1}}.
\end{equation*}
Setting $s = \sigma \left( t \right)$ gives with $c:= c_0 (\alpha+1) \in \cH$
\begin{equation*}
	\frac{\alpha+1}{\sigma \left( t \right)} \frac{1}{\dot{\sigma} \left( t \right)} \dot{z} \left( t \right) + \nabla f \left( z \left( t \right) \right) = \frac{c}{\sigma \left( t \right)^{\alpha+1}}
\end{equation*}
or, equivalently,
\begin{equation*}
	\dot{z} \left( t \right) + \nabla f \left( z \left( t \right) \right) = \frac{c}{t^{\frac{\alpha+1}{2}}}.
\end{equation*}
This is nothing else than the perturbed continuous steepest descent system with perturbation function
$g \colon \left[ t_{0} , + \infty \right[ \to \cH, g(t) = \frac{c}{t^{\frac{\alpha+1}{2}}},$ which obviously fulfills
\begin{equation*}
	\int_{t_{0}}^{+ \infty} \left\lVert g \left( t \right) \right\rVert dt < + \infty \textrm{ and } \int_{t_{0}}^{+ \infty} t \left\lVert g \left( t \right) \right\rVert ^{2} dt < + \infty ,
\end{equation*}
for every $\alpha > 1$. 

Furthermore,
\begin{equation*}
	\int_{t_{0}}^{+ \infty} t^{2} \left\lVert g \left( t \right) \right\rVert ^{2} dt = \int_{t_{0}}^{+ \infty} \dfrac{\left\lVert c \right\rVert ^{2}}{t^{\alpha-1}} dt < + \infty ,
\end{equation*}
whenever $\alpha > 2$.
All statements excepting \ref{rescale-pert:rate-vel} follow from Theorem \ref{theorem SD pert}.

Going back to the perturbed continuous steepest descent \eqref{pert SD}, we see that
\begin{equation*}
	\left\lVert \dot{z} \left( t \right) \right\rVert \leq \left\lVert \nabla f \left( z \left( t \right) \right) \right\rVert + \left\lVert g \left( t \right) \right\rVert = o \left( \dfrac{1}{t} \right) + o \left( \dfrac{1}{t^{\frac{\alpha+1}{2}}} \right) = o \left( \dfrac{1}{t} \right) \textrm{ as } t \to + \infty .
\end{equation*}
Taking $t = \tau \left( s \right) = \frac{s^{2}}{2(\alpha+1)}$, it yields
\begin{equation*}
	\left\lVert \dot{z} \left( \tau \left( s \right) \right) \right\rVert = o \left( \dfrac{1}{s^{2}} \right) \textrm{ as } s \to + \infty,
\end{equation*}
which gives \ref{rescale-pert:rate-vel}, since $ \dot{y}(s) = \dot{\tau} (s)\dot{z} (\tau(s)) = \frac{s}{\alpha+1} \dot{z} (\tau(s))$.
\end{proof}

\begin{remark}
As in in Theorem \ref{Thm-rescale-averaging-implicit}, we see that the convergence of trajectory can be guaranteed for $\alpha >1$,  which is less restrictive than for the Su-Boyd-Cand\'es system,  the trajectory of which is known to convergence for $\alpha >3$. This is another positive effect of the approach that  combines time scaling and averaging.
\end{remark}

\subsection{Combining implicit and explicit Hessian driven damping}
As a starting dynamic we consider {the regularized Newton dynamical system}
\begin{equation}\label{eq:Dyn_Newton_1}
\begin{cases}
	\lambda \dot{z}(t) +  \dot{v}(t) + v(t) = 0
	\vspace{1mm}\\
	v(t)=\nabla f( z(t))
\end{cases}.
\end{equation}
It is a special case of the regularized Newton dynamic 
\begin{equation}\label{Hessian_AS}
\lambda (t) \dot{z} (t) + 
\nabla^2  f (z( t)) \dot{z} ( t)+\nabla f(z(t)) =0.
\end{equation}
This system has been studied by Attouch and Svaiter \cite{ASv}, Attouch, Redont, and Svaiter \cite{ARS}, Attouch, Marques Alves, and Svaiter \cite{AMAS} to solve monotone inclusions (that is, with a general maximally monotone operator $A$ instead of $\nabla f$).
In this approach, the central question is the adjustment of the Levenberg-Marquardt regularization parameter $\lambda(\cdot)$ in front of the velocity term. Indeed, taking $\lim_{t\to +\infty} \lambda(t) =0$ allows to be asymptotically close to the Newton method. Our situation concerns the simpler case $\lambda (t) \equiv \lambda >0$ constant, which fits with the convergence properties proved in these papers, and the fact that this dynamic is well-posed. For precise statements concerning the asymptotic behavior of this system see \cite[Theorem 4.1, Remark 4.2]{ASv} and \cite[Theorem 3.9]{ASv}. For the sake of completeness, we provide below its convergence properties, {that shows some improvements compared to the results in \cite{ASv}.  The proof of Theorem \ref{Theorem regNewton} is provided in the Appendix.}
\begin{theorem}\label{Theorem regNewton}
{Suppose that  $f \colon \cH \to \R$ satisfies $(\mathcal A)$.}	
Let $z: [t_0, +\infty[ \to \cH$ be a solution trajectory of the dynamical system \eqref{eq:Dyn_Newton_1}.
Then the following statements are true:
\begin{enumerate}
	\item \label{regNewton:int-fun} 
	(integral estimate of the values)
	$
	\int_{t_0}^{+\infty} \left( f \left( z \left( t \right) \right) -\inf_{\cH} f \right) dt < +\infty;
	$
	
	\item \label{regNewton:rate-fun} 
	(convergence of values towards minimal value)
	$
	f \left( z \left( t \right) \right) -\inf_{\cH} f = o \left( \frac{1}{t} \right) \textrm{ as } t \to + \infty;
	$
	
	\item \label{regNewton:int-vel-grad} 
	(integral estimates of velocities and gradients)
	$\int_{t_{0}}^{+ \infty} t \left\lVert \dot{z} \left( t \right) \right\rVert ^{2} dt < +\infty$
	and
	$\int_{t_{0}}^{+ \infty} t \left\lVert \dot{v} \left( t \right) \right\rVert ^{2} dt < +\infty,$	
	which lead to
	$
	\int_{t_{0}}^{+ \infty} t \left\lVert v \left( t \right) \right\rVert ^{2} dt = \int_{t_{0}}^{+ \infty} t \left\lVert \nabla f \left( z \left( t \right) \right) \right\rVert ^{2} dt < +\infty;
	$		
	
	\item \label{regNewton:rate-grad-vel} 
	(convergence of velocities and gradients towards zero)	\quad 
	$
	\left\lVert v \left( t \right) \right\rVert = \left\lVert \nabla f \left( z \left( t \right) \right) \right\rVert = o \left( \frac{1}{t} \right) \textrm{ as } t \to + \infty,
	$
	which lead to
	$
	\left\lVert \dot{z} \left( t \right) \right\rVert = \left\lVert \dot{v} \left( t \right) \right\rVert = o \left( \frac{1}{t} \right) \textrm{ as } t \to + \infty;
	$
	
	\item 
	the solution trajectory $z(t)$ converges weakly as $t \to +\infty$, and its limit belongs to $S=\argmin f$.
\end{enumerate}
\end{theorem}

\begin{remark}
In \cite[Theorem 4.1]{ASv}, the authors show that both $\dot{z} \left( \cdot \right)$ and $\dot{v} \left( \cdot \right)$, and consequently $v \left( \cdot \right)$, belong to $\mathbb{L}^{2} \left( \left[ t_{0} , + \infty \right[ \right)$. Since $t \mapsto \left\lVert v \left( t \right) \right\rVert$ is nonincreasing, in \cite[Theorem 3.9]{ASv} it is shown that $\left\lVert v \left( t \right) \right\rVert = o \left( \frac{1}{\sqrt{t}} \right)$ as $t \to + \infty$.  By improving some integral estimates for the velocity and the gradient,  we can provide a faster convergence rate, namely of  $o \left( \frac{1}{t} \right)$ as $t \to + \infty$.
\end{remark}

Let us make the change of time variable $t=\tau(s)= \demi \gamma s^2$ in the dynamic \eqref{eq:Dyn_Newton_1}, where $\gamma$ is a positive constant to be adjusted later. 
Set $y(s):= z(\tau(s))$, $w(s):= v(\tau(s))= \nabla f(y(s) )$, and $s_{0} >0$ be such that $t_{0} = \tau \left( s_{0} \right)$.
On the one hand, by the derivation chain rule, we have for every $s \geq s_0$
\begin{equation}\label{eq:Dyn_Newton_10}
\dot{y} (s)= \dot{\tau}(s) \dot{z}(\tau(s)),  \quad \dot{w} (s)= \dot{\tau}(s) \dot{v}(\tau(s)).
\end{equation}
On the other hand, setting $t=\tau (s)$ in \eqref{eq:Dyn_Newton_1} gives 
\begin{equation}\label{eq:Dyn_Newton_11}
\begin{cases}
	\lambda \dot{z}(\tau(s)) +  \dot{v}(\tau(s)) + v(\tau(s)) = 0
	\vspace{1mm}\\
	v(\tau(s))=\nabla f( z(\tau(s))).
\end{cases}
\end{equation}
According to (\ref{eq:Dyn_Newton_10}), (\ref{eq:Dyn_Newton_11}), and
$ \dot{\tau}(s)= \gamma s $, we obtain
\begin{equation}\label{eq:Dyn_Newton_12}
\begin{cases}
	\frac{\lambda}{cs}\dot{y}(s) + \frac{1}{\gamma s}\dot{w}(s) + w(s) = 0
	\vspace{1mm}\\
	w(s) =\nabla f( y(s))
\end{cases}
\end{equation}
or, equivalently,
\begin{equation}\label{eq:Dyn_Newton_13}
\begin{cases}
	\lambda \dot{y}(s) + \dot{w}(s) + \gamma s  w(s) = 0
	\vspace{1mm}\\
	w(s) =\nabla f( y(s)).
\end{cases}
\end{equation}
The convergence rate becomes 
\begin{equation}\label{eq:Dyn_Newton_14}
f (y(s))- \inf_{\cH} f = o \left( \frac{1}{s^2} \right) \textrm{ as } s \to + \infty.
\end{equation}
So, we have accelerated the dynamic for the values, passing from the convergence rate $1/t$ to $1/s^2$.
{Moreover, by proceeding as in Subsection \ref{sec:explicit}, we can return from \eqref{eq:Dyn_Newton_13}  to the dynamics \eqref{eq:Dyn_Newton_11}.}

Let us now come with the averaging process.
Given $s_{0} >0$, we attach to $y(\cdot)$ the new function $x: [s_{0}, +\infty[ \to \cH$ defined by 
\begin{equation}\label{eq:Dyn_Newton_15}
\dot{x}(s) + \dfrac{\delta}{s}(x(s)-y(s))  = 0 ,
\end{equation}
with $x(s_{0}) =x_0$ given in $\cH$,
and where $\delta$ is a positive coefficient to adjust.
Equivalently
\begin{equation}\label{eq:Dyn_Newton_15_b}
y(s)=  x(s) +  \dfrac{s}{\delta}\dot{x}(s)  .
\end{equation}

Let us formulate \eqref{eq:Dyn_Newton_13} in terms of $x$ by eliminating $y$.
We first  obtain 
\begin{equation}\label{eq:Dyn_Newton_16}
\begin{cases}
	\lambda \left( \dot{x}(s) +  \frac{s}{\delta}\ddot{x}(s) + \frac{1}{\delta}\dot{x}(s) \right)   + \dot{w}(s) + \gamma s  w(s) = 0
	\vspace{1mm}\\
	w(s) =\nabla f \left( x(s) +  \frac{s}{\delta}\dot{x}(s) \right).
\end{cases}
\end{equation}
After reduction we obtain
\begin{equation}\label{eq:Dyn_Newton_17}
\begin{cases}
	\ddot{x}(s) + \frac{\delta+ 1}{s} \dot{x}(s)    + 
	\frac{\delta} {\lambda s} \dot{w}(s) + \frac{\delta \gamma} {\lambda}  w(s) = 0
	\vspace{1mm}\\
	w(s) =\nabla f \left( x(s) +  \frac{s}{\delta}\dot{x}(s) \right) .
\end{cases}
\end{equation}
Take $\delta= \alpha -1$, $\gamma= \frac{\lambda}{\alpha -1}$ in \eqref{eq:Dyn_Newton_17}. We have $\delta \gamma= \lambda$, which gives
\begin{equation}\label{eq:Dyn_Newton_18}
\begin{cases}
	\ddot{x}(s) + \frac{\alpha}{s} \dot{x}(s)    + 
	\frac{\alpha -1} {\lambda s} \dot{w}(s) +  w(s) = 0
	\vspace{1mm}\\
	w(s) =\nabla f\left(  x(s) +  \frac{s}{\alpha -1}\dot{x}(s)\right).
\end{cases}
\end{equation}

\noindent According to the general properties of the averaging process
\begin{equation}\label{eq:Dyn_Newton_15_c}
y(s)=  x(s) +  \frac{s}{\alpha -1}\dot{x}(s),
\end{equation}
which has been already studied, see \eqref{eq:change var24_b}, we obtain the following theorem for a dynamics which combines explicit and implicit Hessian driven damping in a new way.

\begin{theorem}\label{Thm-implicit-explicit-Hessian}
{Suppose that  $f \colon \cH \to \R$ satisfies $(\mathcal A)$.}	
Let $x: [s_0, +\infty[ \to \cH$ be a solution trajectory of	
\begin{equation}\label{eq:Dyn_Newton_20}
	\begin{cases}
		\ddot{x}(s) + \dfrac{\alpha}{s} \dot{x}(s)    + 
		\dfrac{\alpha -1} {\lambda s} \dot{w}(s) +  w(s) = 0
		\vspace{1mm}\\
		w(s) =\nabla f\left(  x(s) +  \dfrac{s}{\alpha -1}\dot{x}(s)\right).
	\end{cases}
\end{equation}
Assume that $\alpha>1$. Then the following statements are true:
\begin{enumerate}					
	\item \label{implicit-explicit-Hessian:int-grad} (integral estimate of the gradients) \quad 
	$
	\int_{s_{0}}^{+\infty} s^3 \left\lVert w \left( s \right) \right\rVert ^{2}  ds < +\infty;
	$
	
	\item \label{implicit-explicit-Hessian:rate-grad} (convergence  of  gradients towards zero) \quad
	$
	\left\lVert w \left( s \right) \right\rVert = o \left( \frac{1}{s^{2}} \right) \textrm{ as } s \to + \infty;
	$		
	
	\item \label{implicit-explicit-Hessian:rate-vel} (convergence of velocities towards zero) \quad
	$
	\left\lVert \dot{x} \left( s \right) \right\rVert = o \left( \frac{1}{s} \right) \textrm{ as } s \to + \infty;
	$		
	
	\item \label{implicit-explicit-Hessian:traj} 
	the solution trajectory $x(s)$ converges weakly as $s \to +\infty$, and its limit belongs to $S=\argmin f$.\vspace{1ex}
	
	If $\alpha>3$, then
	\item \label{implicit-explicit-Hessian:rate-fun} (convergence of values towards minimal value) \quad
	$
	f(x(s)) -\inf_{\cH} f = o \left( \frac{1}{s^{2}} \right) \textrm{ as } s \to + \infty.
	$	
\end{enumerate}
\end{theorem}
\begin{proof}
\noindent Let $x(\cdot)$ be a solution trajectory of \eqref{eq:Dyn_Newton_20} which satisfies the Cauchy data $x(s_0) =x_0, \dot{x}(s_0)=x_1$, and $w \left( s_{0} \right) := \nabla f \left( x_{0} + \frac{s_{0}}{\alpha-1} x_{1} \right)$.
In the same way as in the proof of Theorem \ref{Thm-rescale-averaging-implicit}, we first show that such a solution is reached by considering first the solution $y(\cdot)$ of the system
\begin{equation}\label{eq:Dyn_Newton_24}
	\begin{cases}
		\lambda \dot{y}(s) + \dot{w}(s) + \frac{\lambda s}{\alpha -1}  w(s) = 0
		\vspace{1mm}\\
		w(s) =\nabla f( y(s)) \\
		y(s_0)= x_0 + \frac{s_0}{\alpha -1}x_1 ,
	\end{cases}
\end{equation}
and then
\begin{equation}\label{eq:Dyn_Newton_25}
	\begin{cases}
		\dot{x}(s) + \frac{\alpha -1}{s}(x(s)-y(s))  = 0 
		\vspace{2mm}\\
		x(t_0)= x_0 .
	\end{cases}
\end{equation}
Indeed, \eqref{eq:Dyn_Newton_25} gives $\dot{x}(s_0) = \frac{\alpha -1}{s_0}(y(s_{0})-x(s_{0}))$, which, by the second equation of \eqref{eq:Dyn_Newton_24}, equals to $x_1$.	
We have already seen that $x$ can be seen as an averaging process \eqref{eq:Dyn_Newton_25} of $y$ as follows
\begin{equation}\label{eq:Dyn_Newton_26}
	x(s) =   \int_{s_0}^s y(u)\,  d\mu_{s} (u) + \xi (s),
\end{equation}
where $\mu_s$ is the measure on $[s_0, s]$ defined by
$$
\mu_s = \frac{ s_0^{\alpha -1}}{ s^{\alpha -1}} \delta_{s_0} +  (\alpha -1) \frac{u^{\alpha -2}}{s^{\alpha -1}} du,
$$
{where} $\delta_{s_0}$ is the Dirac measure at $s_0$ and
\vspace*{-10pt}
\begin{equation*}
	\xi (s):=- \frac{{s_0}^{\alpha}}{(\alpha -1)s^{\alpha -1}} x_1.
\end{equation*}
The statements \ref{implicit-explicit-Hessian:int-grad} and \ref{implicit-explicit-Hessian:rate-grad} follow from Theorem \ref{Theorem regNewton} after time rescaling and integration by substitution. In order to show \ref{implicit-explicit-Hessian:traj}, we need to pass from the convergence of $y$ to that of $x$ by using the interpretation of $x$ as an average of $y$ plus a negligible term. Moreover, we know that $y$ converges weakly to an element in $S$ thanks to Theorem \ref{Theorem regNewton}. The convergence of $x$ follows from similar arguments as in Theorem \ref{Thm-rescale-averaging-implicit}.

Finally, since the averaging process is the same as in Theorem \ref{Thm-rescale-averaging-implicit}, the convergence rate of $1/s^{2}$ for the values \ref{implicit-explicit-Hessian:rate-fun} follows by the same arguments.
\end{proof}

\subsection{General Hessian damping coefficient{: extension to bilevel convex optimization}}
Given $\beta_0$ a positive damping coefficient, consider the more general form of the dynamic
\begin{equation}\label{general_damping_1}
\ddot{x}(s) + \frac{\alpha}{s}\dot{x}(s) + \nabla f\left(x(s)+  \beta_0 \frac{s}{\alpha -1}\dot{x}(s) \right)  = 0.
\end{equation}
Our study in the previous subsections concerned the case $\beta_0 = 1$.   This section shows that the time rescaling and averaging technique links the dynamical system \eqref{general_damping_1} with a bilevel convex optimization problem and thus emphasizes the versatility of the approach proposed in this paper.

Let us introduce 
\begin{equation}\label{general_damping_2}
y(s)\eqdef x(s)+  \beta_0 \frac{s}{\alpha -1}\dot{x}(s), 
\end{equation}
so that \eqref{general_damping_1} is written equivalently 
\begin{equation}\label{general_damping_3}
\ddot{x}(s) + \frac{\alpha}{s}\dot{x}(s) + \nabla f\left(y(s)\right)  = 0.
\end{equation}
Let us reformulate \eqref{general_damping_3} as a differential equation with $y(s)$ as the state variable. According to \eqref{general_damping_2} we have 
\begin{equation}\label{general_damping_4}
\dot{x}(s)= \dfrac{\alpha -1}{\beta_0 s}(y(s)-x(s)) ,
\end{equation}
which after derivation gives
\begin{equation}\label{general_damping_5}
\ddot{x}(s)= \dfrac{\alpha -1}{\beta_0 s}(\dot{y}(s)-\dot{x}(s)) -\dfrac{\alpha -1}{\beta_0 s^2}(y(s)-x(s)) .
\end{equation}
Combining \eqref{general_damping_3} and \eqref{general_damping_5} we obtain
\begin{equation}\label{general_damping_6}
\dfrac{\alpha -1}{\beta_0 s}(\dot{y}(s)-\dot{x}(s)) -\dfrac{\alpha -1}{\beta_0 s^2}(y(s)-x(s)) + \frac{\alpha}{s}\dot{x}(s) + \nabla f\left(y(s)\right)  = 0.
\end{equation}
By replacing $\dot{x}(s)$ by its formulation given in \eqref{general_damping_4} we obtain 
$$
\dfrac{\alpha -1}{\beta_0 s}\dot{y}(s) + \left(  \frac{\alpha}{s} -\dfrac{\alpha -1}{\beta_0 s}\right)\dfrac{\alpha -1}{\beta_0 s} (y(s)-x(s)) - \frac{\alpha -1}{\beta_0 s^2}(y(s)-x(s)) + \nabla f\left(y(s)\right)  = 0.
$$
\noindent After reduction we get
\begin{equation}\label{general_damping_8}
\dfrac{\alpha -1}{\beta_0 s}\dot{y}(s) +\frac{(\alpha -1)^2 (\beta_0 -1)}{\beta_0^2 s^2}(y(s)-x(s)) + \nabla f\left(y(s)\right)  = 0.
\end{equation}
Equivalently
\begin{equation}\label{general_damping_9}
\dot{y}(s) + \dfrac{\beta_0 s}{\alpha -1} \nabla f\left(y(s)\right) +\frac{(\alpha -1) (\beta_0 -1)}{\beta_0 s}(y(s)-x(s))  = 0.
\end{equation}
Putting together \eqref{general_damping_4} and \eqref{general_damping_9} we obtain the  system
\begin{equation}\label{general_damping_10}
\begin{cases}
	\dot{y}(s) + \dfrac{\beta_0 s}{\alpha -1} \nabla f\left(y(s)\right) +\dfrac{(\alpha -1) (\beta_0 -1)}{\beta_0 s}(y(s)-x(s))  = 0 \vspace{2mm}
	\\
	\dot{x}(s) + \dfrac{\alpha -1}{\beta_0 s}(x(s)-y(s))  = 0 .
\end{cases}
\end{equation}

\noindent Let us now consider time scaling of the above system {as in \eqref{change var275b}}.
Set 
$$s= \sqrt{2(\alpha -1)t} \, \mbox{ 	and }  \, y (\sqrt{2(\alpha -1)t})=Y(t), \, x (\sqrt{2(\alpha -1)t})=X(t).  $$

A similar calculation as in Subsection \ref{sec:time_scale}  gives
\begin{equation}\label{general_damping_11}
\begin{cases}
	\dot{Y}(t) + \beta_0 \nabla f\left(Y(t)\right) +\dfrac{(\alpha -1) (\beta_0 -1)}{2\beta_0 t}(Y(t)-X(t))  = 0 \vspace{2mm}
	\\
	\dot{X}(t) + \dfrac{\alpha -1}{2\beta_0 t}(X(t)-Y(t ))  = 0 .
\end{cases}
\end{equation}
This is a perturbation of the steepest descent dynamical system in the product space $\cH \times \cH$. Precisely, setting $Z(t)= (Y(t), X(t)) \in \cH \times \cH$, we have
\begin{equation}\label{general_damping_12}
\dot{Z}(t) + \beta_0  \nabla \Psi \left(Z(t)\right) +\frac{(\alpha -1) }{2\beta_0 t}A(Z(t))  = 0,
\end{equation}
where $\Psi (Z)= f(Y)$ and $A:\cH \times \cH \to \cH \times \cH$ is the linear operator whose matrix is given by
\begin{center}
$A= $$\begin{pmatrix}
	\beta_0 -1 & -(\beta_0 -1)  \\
	-1 & +1 
\end{pmatrix}$$
$
\end{center}
The case $\beta_0=2$ is of particular interest since then
$A$ is  a symmetric operator which is positive semidefinite. The study of this system then follows from classical results concerning gradient dynamical systems with multiscale aspects, see for example the work of Attouch and Czarnecki \cite{AC}.

The dynamical system \eqref{general_damping_12} for $\beta_0 = 2$ fits in the dynamics
\begin{equation}\label{general_damping_13}
\dot{Z}(t) + \beta_0  \nabla \Psi \left(Z(t)\right) + \varepsilon \left( t \right) \nabla \Phi \left( Z(t) \right)  = 0,
\end{equation}
with $\varepsilon \left( t \right) := \frac{(\alpha -1) }{2\beta_0 t}$ and $\Phi \left( \cdot \right) = \frac{1}{2} \left\lVert \cdot \right\rVert _{A}^{2}$, where $\left\lVert \cdot \right\rVert _{A}$ is the seminorm induced by the symmetric and positive semidefinite operator $A$.
The differential equation \eqref{general_damping_13} can be seen as a continuous time approach of the bilevel optimization problem
\begin{equation}\label{general_damping_14}
\min_{\cH \times \cH} \left\lbrace \Phi \left( Z \right) \ | \ Z \in D := \argmin_{\cH \times \cH} \Psi  \right\rbrace .
\end{equation}
Let $Z_{*}$ be a solution of \eqref{general_damping_14}, meaning that $\Phi \left( Z_{*} \right) = \inf_{D} \Phi$ and $\Psi \left( Z_{*} \right) = \inf_{\cH \times \cH} \Psi$.
Define for $t \geq t_{0}$
\begin{equation*}
\Psi_{\Phi} \left( t \right) := \beta_0 \left( \Psi \left(Z(t)\right) - \inf\nolimits_{\cH \times \cH} \Psi \right) + \varepsilon \left( t \right) \left( \Phi \left(Z(t)\right) - \inf\nolimits_{D} \Phi \right) \geq 0 .
\end{equation*}
Indeed the above quantity is nonnegative because we are in the particular simple case where $\Phi(Y,X)= \demi\|Y-X\|^2$ is nonnegative, $D= \argmin_{\cH} f  \times \cH$ and hence $\inf_{D} \Phi =0$.
Let us compute
\begin{align*}
\dfrac{d}{dt} \left( \dfrac{1}{2} \left\lVert Z \left( t \right) - Z_{*} \right\rVert ^{2} \right) & = \left\langle Z \left( t \right) - Z_{*} , \dot{Z} \left( t \right) \right\rangle = - \beta_0 \left\langle Z \left( t \right) - Z_{*} ,  \nabla \Psi \left(Z(t)\right) \right\rangle - \varepsilon \left( t \right) \left\langle Z \left( t \right) - Z_{*} , \nabla \Phi \left( Z(t) \right) \right\rangle \nonumber \\
& \leq - \beta_0 \left( \Psi \left(Z(t)\right) - \inf\nolimits_{\cH \times \cH} \Psi \right) - \varepsilon \left( t \right) \left( \Phi \left(Z(t)\right) - \inf\nolimits_{D} \Phi \right) = - \Psi_{\Phi} \left( t \right) .
\end{align*}
This shows that $\lim_{t \to + \infty} \left\lVert Z \left( t \right) - Z_{*} \right\rVert \in \R$ exists and that
$$
\int_{t_{0}}^{+ \infty} \Psi_{\Phi} \left( t \right) dt < + \infty.
$$
This yields $\liminf_{t \rightarrow +\infty} t \Psi_{\Phi} \left( t \right) =0$.
On the other hand, we have
\begin{align*}
\dfrac{d}{dt} \Psi_{\Phi} \left( t \right)
& = \beta_0 \left\langle \nabla \Psi \left( Z \left( t \right) \right) , \dot{Z} \left( t \right) \right\rangle + \varepsilon \left( t \right) \left\langle \nabla \Phi \left( Z \left( t \right) \right) , \dot{Z} \left( t \right) \right\rangle + \dot{\varepsilon} \left( t \right) \left( \Phi \left(Z(t)\right) - \inf_{D} \Phi \right) 
\leq - \left\lVert \dot{Z} \left( t \right) \right\rVert ^{2} .
\end{align*}
This means that $\Psi_{\Phi}$ is decreasing and furthermore
\begin{equation*}	
\dfrac{d}{dt} \left( t \Psi_{\Phi} \left( t \right)  \right) = \Psi_{\Phi} \left( t \right) + t \dfrac{d}{dt} \Psi_{\Phi} \left( t \right) \leq \Psi_{\Phi} \left( t \right) - t \left\lVert \dot{Z} \left( t \right) \right\rVert ^{2}.
\end{equation*}
Since $\Psi_{\Phi} \in \mathbb{L}^{1} \left( \left[ t_{0} , + \infty \right[ \right)$, we conclude that $t \| \dot{Z} \left( t \right) \| ^{2} \in \mathbb{L}^{1} \left( \left[ t_{0} , + \infty \right[ \right)$ and furthermore $\lim_{t \to + \infty} t \Psi_{\Phi} \left( t \right) = 0$.
Consequently,
\begin{equation*}
\lim_{t \to + \infty} t \left( \Psi \left(Z(t)\right) - \inf\nolimits_{\cH \times \cH} \Psi \right) = 0 \textrm{ and } \lim_{t \to + \infty} \left( \Phi \left(Z(t)\right) -\inf\nolimits_{D} \Phi \right) = 0 .
\end{equation*}
Since $\Psi$ and $\Phi$ are convex and lower semicontinuous, the second condition of Opial's lemma is fulfilled.  Thus $Z(t)$ converges weakly to a solution of \eqref{general_damping_14} as $t \rightarrow +\infty$.

Moreover, according to the definition of $\Phi$ and $A$, it holds
\begin{equation*}
\Phi \left(Z(t)\right) = \dfrac{1}{2} \left\langle Z(t) , A \left( Z \left( t \right) \right) \right\rangle = \dfrac{1}{2} \left\lVert X(t) - Y(t) \right\rVert ^{2} = \dfrac{2\beta_0^{2}}{\left( \alpha-1 \right) ^{2}} t^{2} \left\lVert \dot{X} (t) \right\rVert ^{2} \to 0 \ \mbox{as} \ t \rightarrow +\infty.
\end{equation*}
According to the definitions of $\Psi$ and $\Phi$
we conclude that there exists $x_* \in \argmin_{\cH} f$ such that $(Y(t), X(t))$ converges weakly to $(x_*,x_*)$ as  $t \rightarrow +\infty$,
$$f(Y(t)) - \inf\nolimits_{\cH} f = o\left(\frac{1}{t} \right) \ \mbox{and} \ \lim_{t \rightarrow +\infty} \|X(t) - Y(t)\| =  \lim_{t \rightarrow +\infty} t \|\dot X(t)\| = 0.$$
This means that $x(s)$ converges weakly to $x_*$ as $s \rightarrow +\infty$,  
$$f\left(x(s)+  \frac{2s}{\alpha -1}\dot{x}(s)\right) - \inf\nolimits_{\cH} f = o\left(\frac{1}{s^2} \right) \ \mbox{and} \  \lim_{s \rightarrow +\infty} s \|\dot x(s)\| = 0.$$

In the following theorem we collect the convergence properties of the trajectory of the dynamical system \eqref{general_damping_1} in case $\beta_0=2$.

\begin{theorem}\label{Thm-general-Hessian-coefficient}
{Suppose that  $f \colon \cH \to \R$ satisfies $(\mathcal A)$.}	
Let $x: [s_0, +\infty[ \to \cH$ be a solution trajectory of	
\begin{equation*}
	\ddot{x}(s) + \frac{\alpha}{s}\dot{x}(s) + \nabla f\left(x(s)+  \frac{2s}{\alpha -1}\dot{x}(s) \right)  = 0.
\end{equation*}
Assume that $\alpha>1$. Then the following statements are true:
\begin{enumerate}		
	\item \label{general-Hessian-coefficient:rate-fun} (convergence of values towards minimal value) \quad
	$
	f\left(x(s)+  \frac{2s}{\alpha -1}\dot{x}(s)\right)  -\inf_{\cH} f = o \left( \frac{1}{s^{2}} \right) \textrm{ as } s \to + \infty;
	$	
	
	\item \label{general-Hessian-coefficient:rate-vel} (convergence of velocities towards zero) \quad
	$
	\left\lVert \dot{x} \left( s \right) \right\rVert = o \left( \frac{1}{s} \right) \textrm{ as } s \to + \infty;
	$		
	
	\item \label{general-Hessian-coefficient:traj} 
	the solution trajectory $x(s)$ converges weakly as $s \to +\infty$, and its limit belongs to $S=\argmin f$.
	
\end{enumerate}
\end{theorem}

\section{Nonsmooth case}\label{sec:nonsmooth}
In this section, we show that the approach proposed in this paper extends to the case when $f:\cH \to \rinf$ is assumed to be a (possibly nonsmooth) proper,  convex and lower semicontinuous function.

\subsection{The continuous dynamics}
As a basic property {of the dynamical system \eqref{change var27_thm}}, the couple of variables $(x,y)$ defined in Subsection \ref{subsec:rates} satisfies the following differential system, which involves only first-order derivatives in time and space 
\begin{equation}\label{eq:fos2}
\begin{cases}
	\dot{y}(s) + \frac{s}{\alpha -1}\nabla f( y(s)) & = 0
	\vspace{2mm}\\
	\dot x(s)+ \frac{\alpha -1}{s} (x(s) -y(s)) & = 0.
\end{cases}
\end{equation}
\noindent This naturally suggests the extension of the above results to the nonsmooth case (replace  the gradient of $f$ by its subdifferential).
We thus get the differential inclusion system
\begin{equation}\label{eq:fos2b}
\begin{cases}
	\dot{y}(s) + \frac{s}{\alpha -1}\partial f( y(s)) & \ni 0
	\vspace{2mm}\\
	\dot x(s)+ \frac{\alpha -1}{s} (x(s) -y(s)) & = 0.
\end{cases}
\end{equation}

Solving this system gives generalized solutions to the second-order differential inclusion
\begin{equation}\label{eq:fos2b2}
\ddot{x}(s) + \frac{\alpha}{s}\dot{x}(s) + \partial f\left(x(s)+   \frac{s}{\alpha -1}\dot{x}(s) \right)  \ni 0
\end{equation}
whose direct study raises several difficulties, {recalling that $x(s_0) = x_0$ and $\dot{x}(s_0) = x_1$}.
To extend the results  of the previous section to this nonsmooth case, we need to avoid the arguments using the Lipschitz continuity of $\nabla f$.
So, we are led to consider the Cauchy problem for \eqref{eq:fos2b2} with 
initial data $x_0 \in \dom f$ and $x_1=0$, that is with initial velocity equal to zero.
So doing we have $y({s_0})= x({s_0})=x_0$, which allows to interpret $x$ as an average of $y$ without correcting term ($\xi=0$ with the notations of Subsection 2.6).

The  existence and uniqueness of a strong solution to the associated Cauchy problem relies on the equivalent formulation of \eqref{eq:fos2b} as a perturbation of the  generalized steepest descent dynamical system in the product space $\cH \times \cH$. Precisely, define $F:  \cH \times \cH \to \rinf$, for any $Z=(y,x) \in \cH \times \cH$ by
\vspace*{-10pt}
\begin{equation*}
F(Z) = f(y),
\end{equation*}
and let  $G: \cH \times \cH \to \cH \times \cH $ be the operator defined by
\begin{equation}
G(Z)= (0,  x -y) .
\end{equation}
Then \eqref{eq:fos2b} is written equivalently as
\begin{equation}
\dot{Z}(s) + \frac{s}{\alpha -1}\partial F (Z(s)) + \frac{\alpha -1}{s} G(Z(s))\ni 0.
\end{equation}
The initial condition becomes $Z(t_0)= (x_0,x_0)$ which belongs to 
$\dom F= \dom f \times \cH$.
According to the classical results concerning the Lipschitz perturbation of evolution equations governed by subdifferentials of convex functions,  see  \cite[Proposition 3.12]{breziseqevo}, we obtain the existence and uniqueness of a global strong solution of the Cauchy problem associated with \eqref{eq:fos2b}.
As a major advantage of the time scaling and averaging techniques, the  arguments used in the previous section still work in this more general nonsmooth situation. The rules of differential calculus are still valid for strong solutions, see \cite[chapter VIII.2]{brezisAF}, and Jensen's inequality is still valid for a nonsmooth function $f$ {(see e.g. \cite[Proposition 9.24]{BC})}. Indeed Jensen's inequality is classical for a smooth convex function $f$. Its extension to the nonsmooth convex case can be obtained by first writing it for the Moreau-Yosida regularization $f_{\lambda}$ of $f$, then passing to the limit when $\lambda  \downarrow 0$. According to the monotone convergence of   $f_{\lambda}$ towards $f$, we can pass to the limit in the integral term thanks to the Beppo Levi monotone convergence theorem.
Therefore we obtain the following theorem.
\vspace*{-5pt}
\begin{theorem}\label{Thm-nonsmooth}
Let $f:\cH \to \rinf$ be a proper, lower semicontinuous, and convex function such that $S=\argmin f \neq \emptyset$. 
Let $x: [s_{0}, +\infty[ \to \cH$ be a solution trajectory of 
\begin{equation}\label{change var27_thm_nonsmooth}
	\ddot{x}(s) + \frac{\alpha}{s}\dot{x}(s) + \partial f\left(x(s)+   \frac{s}{\alpha -1}\dot{x}(s) \right)  \ni 0 ,
\end{equation}
which satisfies the initial conditions $x(s_{0}) \in \dom f$ and $\dot{x}(s_{0}) =0$.
Assume that $\alpha >1$. Then 
\begin{enumerate}
	\item the solution trajectory $x(s)$ converges weakly as $s \to +\infty$, and its limit belongs to $S=\argmin f$. \vspace{1ex}
	
	If $\alpha >3$, then
	\item 
	(convergence of values towards minimal value) \quad
	$
	f \left( y \left( s \right) \right) -\inf_{\cH} f = o \left( \frac{1}{s^{2}} \right) \textrm{ as } s \to + \infty.
	$
\end{enumerate}	
\end{theorem}

\begin{remark}
As a consequence of $ii)$, we have that $x(s)$ remains in the domain of $f$ for all $s\geq t_0$. This viability property strongly depends on the fact that the initial position belongs to the domain of $f$,  and that the initial velocity has been taken equal to zero.
\end{remark}

\begin{remark}
The  above theorem is valid in an infinite dimensional setting, which makes it applicable to nonlinear PDEs. 
According to the classical theory for the  steepest descent dynamic in the nonsmooth convex case, we have that at each $s>s_0$  the right derivative of 
$y(s)$ exists and satisfies 
$$
-\left(\dfrac{dy}{ds}\right)^{+}(s)= \frac{s}{\alpha -1}\left(\partial f( y(s))\right)^0
$$
where $\left(\partial f( y(s))\right)^0$ is the element of minimum norm of $\partial f( y(s))$.
Thus, $y$ may exhibit shocks. By contrast, $x$, which is an average of $y$,
has a continuous derivative, hence does not exhibit shocks. Still $x$ is not twice differentiable.
\end{remark}

\subsection{A proximal type algorithm for nonsmooth minimization}
\smallskip
Recall the second-order differential inclusion \eqref{eq:fos2b2} associated with  \eqref{basic-min} in Section \ref{sec:nonsmooth} 
\begin{equation}\label{prox-al:01}
\ddot{x}(s) + \frac{\alpha}{s}\dot{x}(s) + \partial f\left(x(s)+   \frac{s}{\alpha -1}\dot{x}(s) \right)  \ni 0,
\end{equation}
or equivalently formulated  as a first order system (see \eqref{eq:fos2b})
\begin{equation}\label{prox-al:02}
\begin{cases}
	\dot{y}(s) + \frac{s}{\alpha -1}\partial f( y(s)) & \ni 0
	\vspace{2mm}\\
	\dot x(s)+ \frac{\alpha -1}{s} (x(s) -y(s)) & = 0.
\end{cases}
\end{equation}
We consider the following implicit discretization in time of  \eqref{prox-al:02}, written for every $k \geq 0$ as
\begin{equation}\label{prox-al:03}
\begin{cases}
	y_{k+1} - y_{k} + \frac{s_{k+1}}{\alpha - 1} \partial f \left( y_{k+1} \right) & \ni 0
	\vspace{2mm}\\
	x_{k+1} - x_{k} + \frac{\alpha - 1}{s_{k+1}} \left( x_{k} - y_{k+1} \right) & = 0 .
\end{cases}
\end{equation}
This gives rise to the following numerical algorithm
\begin{equation}\label{prox-al:scheme}
(\forall k \geq 0) \quad \begin{cases}
	y_{k+1} & := \prox_{\frac{s_{k+1}}{\alpha - 1} f} \left( y_{k} \right) 
	\vspace{1mm}\\
	x_{k+1}	& := \left( 1 - \frac{\alpha - 1}{s_{k+1}} \right) x_{k} + \frac{\alpha - 1}{s_{k+1}} y_{k+1},
\end{cases}
\end{equation}
where $y_0, x_0 \in \cH$ are arbitrary initial points. This algorithm has some analogy with the relaxed inertial proximal algorithm for maximally monotone operators in \cite{AC-RIPA}.

For the step size sequence $\left( s_{k} \right) _{k \geq 0}$ we impose the following recurrence condition
\begin{equation}\label{prox-al:sk:rec}
s_{0} := 0 \qquad \textrm{ and } \qquad s_{k+1}^{2} - \left( \alpha - 1 \right) s_{k+1} = s_{k}^{2} \quad \forall k \geq 0, 
\end{equation}
which relates to the Nesterov step size rule. Indeed, by denoting $t_{k} := \frac{s_{k}}{\alpha - 1}$, we have for every $k \geq 0$
\begin{equation*}
s_{k+1} := \dfrac{\alpha - 1 + \sqrt{\left( \alpha - 1 \right) ^{2} + 4s_{k}^{2}}}{2} \ \mbox{or, equivalently,} \ t_{k+1} := \dfrac{1 + \sqrt{1 + 4t_{k}^{2}}}{2},
\end{equation*}
which is nothing but the classic step size rule of  Nesterov's accelerated gradient method \cite{Nest1,Nest2,BT}.

We will nevertheless continue to work with the sequence $\left( s_{k} \right) _{k \geq 0}$ in order to emphasize that all the convergence statements provided in this section are valid for $\alpha > 1$,  contrary  to the stronger condition $\alpha > 3$ for the low resolution case (see Remark \ref{remark alpha>1}). 
We have that $s_{k} \sim k$, precisely, $s_{k} \geq \left( k+1 \right) \left( \alpha-1 \right) / 2$ for every $k \geq 0$ (cf. \cite[Lemma 4.3]{BT}).
Furthermore, by a telescoping sum argument, we have from  \eqref{prox-al:sk:rec}  
$$
s_{k+1}^{2} - \left( \alpha - 1 \right) \sum_{i=0}^{k} s_{i+1} = s_{0}^{2} .
$$
Since $s_{0} = 0$, we can conclude that for every $k \geq 0$
\begin{equation}\label{prox-al:sk}
s_{k}^{2} = \left( \alpha - 1 \right) \sum_{i=0}^{k} s_{i}.
\end{equation}
Let us also notice that \eqref{prox-al:scheme} can be written only in terms of the sequence $\left( x_{k} \right) _{k \geq 0}$
$$
x_{k+1} - x_{k} - \dfrac{s_{k} - \left( \alpha - 1 \right)}{s_{k+1}} \left( x_{k} - x_{k-1} \right) + \partial f \left( x_{k} + \dfrac{s_{k+1}}{\alpha - 1} \left( x_{k+1} - x_{k} \right) \right) \ni 0 \quad \forall k \geq 1,
$$
which can be seen as a direct discretization of \eqref{prox-al:01}.

In the following we first analyze the convergence properties of the sequence $\left( y_{k} \right) _{k \geq 0}$ generated by \eqref{prox-al:scheme}. Then, we transfer these to the sequence $\left( x_{k} \right) _{k \geq 0}$. One can easily observe the similarity between this approach and the continuous time one, where we transferred the converge properties of the trajectory $y(\cdot)$ generated by the continuous steepest descent system after time scaling to the averaged trajectory  $x(\cdot)$.
\begin{theorem}\label{thm:prox1}
Let $\left( y_{k} \right) _{k \geq 0}$ be the sequence generated by \eqref{prox-al:scheme}.
The following statements are true:
\begin{enumerate}
	\item (summability of function values) \quad
	$\sum_{k \geq 0} s_{k} \left( f \left( y_{k} \right) - \inf_{\cH} f \right) < + \infty ;$
	
	\item (summability of subgradients)
	there exists a sequence $\left( \eta_{k} \right) _{k \geq 0} \subseteq \cH$ such that $\eta_{k} \in \partial f \left( y_{k} \right)$ for every $k \geq 0$ and
	$\sum_{k \geq 0} s_{k}^{2} \left\lVert \eta_{k} \right\rVert ^{2}$ $< + \infty .$
	
	In particular, if $f$ is differentiable, then \quad
	$
	\left\lVert \nabla f \left( y_{k} \right) \right\rVert = {o \left( \frac{1}{\sqrt{\sum_{i = 0}^k s_{i}^{2}}} \right) = o \left( \frac{1}{k\sqrt{k}} \right)} \textrm{ as } k \to + \infty ;
	$
	
	\item (convergence of values towards minimal value) \quad
	$f \left( y_{k} \right) - \inf_{\cH} f = \mathcal{O} \left( \frac{1}{\sum_{i=0}^{k} s_{i}} \right) = o \left( \frac{1}{s_{k}^{2}} \right) = o \left( \frac{1}{\left( k+1 \right) ^{2}} \right)$ as $k \to + \infty ;$			
	
	\item the sequence of iterates $\left( y_{k} \right) _{k \geq 0}$ converges weakly as $k \to + \infty$, and its limit belongs to $S = \argmin_{\cH} f$.
\end{enumerate}
\end{theorem}

\begin{proof}
Let $k \geq 0$ be fixed. Take $z_{*} \in S= \argmin f$. According to \eqref{prox-al:scheme} there exists $\eta_{k+1} \in \partial f \left( y_{k+1} \right)$ such that $y_{k+1} - y_{k} + \frac{s_{k+1}}{\alpha - 1} \eta_{k+1} = 0$. By  convexity of $f$ we deduce that
\begin{align}
	\dfrac{1}{2} \left\lVert y_{k+1} - z_{*} \right\rVert ^{2} & = \dfrac{1}{2} \left\lVert y_{k} - z_{*} \right\rVert ^{2} + \left\langle y_{k+1} - z_{*} , y_{k+1} - y_{k} \right\rangle - \dfrac{1}{2} \left\lVert y_{k+1} - y_{k} \right\rVert ^{2} \nonumber \\
	& = \dfrac{1}{2} \left\lVert y_{k} - z_{*} \right\rVert ^{2} - \dfrac{s_{k+1}}{\alpha - 1} \left\langle y_{k+1} - z_{*} ,  \eta_{k+1} \right\rangle - \dfrac{1}{2} \dfrac{s_{k+1}^{2}}{\left( \alpha - 1 \right) ^{2}} \left\lVert \eta_{k+1} \right\rVert ^{2} \nonumber \\
	& \leq \dfrac{1}{2} \left\lVert y_{k} - z_{*} \right\rVert ^{2} - \dfrac{s_{k+1}}{\alpha - 1} \left( f \left( y_{k+1} \right) - \inf_{\cH} f \right) - \dfrac{1}{2} \dfrac{s_{k+1}^{2}}{\left( \alpha - 1 \right) ^{2}} \left\lVert \eta_{k+1} \right\rVert ^{2} . \label{prox-al:04}
\end{align}
The statements $i)$ and $ii)$ follow from \cite[Lemma 5.31]{BaCo}. In addition, the limit $\lim_{k \to + \infty} \left\lVert y_{k} - z_{*} \right\rVert \in \R$ exists, which means that the first condition of the discrete Opial's lemma is fulfilled. 

The convergence rate in $ii)$ follows from the fact that sequence $\left(  \eta_{k} \right) _{k \geq 0}$ is nonincreasing.  Indeed, we have from the monotonicity of $\partial f$
\begin{align*}
	\left\lVert \eta_{k} \right\rVert^{2} - \left\lVert \eta_{k+1} \right\rVert^{2} & = - 2 \left\langle \eta_{k+1} , \eta_{k+1} - \eta_{k} \right\rangle + \left\lVert \eta_{k+1} - \eta_{k} \right\rVert^{2}
	= \dfrac{\alpha - 1}{s_{k+1}} \left\langle y_{k+1} - y_{k} , \eta_{k+1} - \eta_{k} \right\rangle + \left\lVert \eta_{k+1} - \eta_{k} \right\rVert^{2} \geq 0 ,
\end{align*}
and the rate follows from  \cite[Lemma 22]{AC2}. Similarly, we notice that
\color{black}
the sequence $\left( f \left( y_{k} \right) - \inf_{\cH} f \right) _{k \geq 0}$ is nonincreasing as well. Precisely, for every $k \geq 0$ we have
$$
\left(f \left( y_{k} \right) - \inf_{\cH} f \right) -  \left(f \left( y_{k+1} \right) - \inf_{\cH} f \right) \geq \langle \eta_{k+1}, y_k-y_{k+1} \rangle  = \dfrac{\alpha - 1}{s_{k+1}} \left\lVert y_{k+1} - y_{k} \right\rVert ^{2} \geq 0.
$$
According to \cite[Lemma 22]{AC2} we get
\begin{equation*}
	f \left( y_{k} \right) - \inf_{\cH} f = o \left( \dfrac{1}{\sum_{i=0}^{k} s_{i}} \right) ,
\end{equation*}
which proves $iii)$, thanks to \eqref{prox-al:sk}.
Finally, since $\lim_{k \to + \infty} f \left( y_{k} \right) = \inf_{\cH} f$ and $f$ is convex and lower semicontinuous, the second condition of Opial's lemma is also fulfilled. This gives the weak convergence of the sequence $\left( y_{k} \right) _{k \geq 0}$  to an element in $S=\argmin f$.
\end{proof}

Let us consider the convergence properties of  $(x_k)_{k \geq 0}$ which are associated with the averaging process.
\begin{theorem}
Let $\left( x_{k} \right) _{k \geq 0}$ be the sequence generated by \eqref{prox-al:scheme}.
The following statements are true:
\begin{enumerate}
	\item (convergence of values) \quad
	$f \left( x_{k} \right) - \inf_{\cH} f = \mathcal{O} \left( \frac{1}{\sum_{i=0}^{k} s_{i}} \right) = \mathcal{O} \left( \frac{1}{s_{k}^{2}} \right) = \mathcal{O} \left( \frac{1}{\left( k+1 \right) ^{2}} \right)$ as $k \to + \infty ;$
	
	\item the sequence of iterates $\left( x_{k} \right) _{k \geq 0}$ converges weakly as $k \to + \infty$, and its limit belongs to $S = \argmin_{\cH} f$.
\end{enumerate}
\end{theorem}

\begin{proof}
We begin by interpreting $\left( x_{k} \right) _{k \geq 0}$ as an average of $\left( y_{k} \right) _{k \geq 0}$. Multiplying the second equation in \eqref{prox-al:scheme} by $s_{k+1}^{2}$ and using \eqref{prox-al:sk:rec}, we obtain for every $k \geq 0$
$$
s_{k+1}^{2} x_{k+1} = \left( s_{k+1}^{2} - \left( \alpha - 1 \right) s_{k+1} \right) x_{k} + \left( \alpha - 1 \right) s_{k+1} y_{k+1} = s_{k}^{2} x_{k} + \left( \alpha - 1 \right) s_{k+1} y_{k+1} .
$$
By a telescopic sum argument and using \eqref{prox-al:sk}, we obtain from here the following expression for each $k \geq 0$
$$
x_{k+1} = \dfrac{1}{s_{k+1}^{2}} \sum_{i=0}^{k+1} \left( \alpha - 1 \right) s_{i} y_{i} = \dfrac{1}{\sum_{i=0}^{k+1} s_{i}} \sum_{i=0}^{k+1} s_{i} y_{i} .
$$
Then, Jensen's inequality gives for every $k \geq 0$
\begin{align}
	f \left( x_{k+1} \right) - \inf_{\cH} f & =  \left( f - \inf_{\cH} f \right) \left( \dfrac{1}{\sum_{i=0}^{k+1} s_{i}} \sum_{i=0}^{k+1} s_{i} y_{i} \right) \nonumber \\
	& \leq  \dfrac{1}{\sum_{i=0}^{k+1} s_{i}} \sum_{i=0}^{k+1} s_{i} \left( f \left( y_{i} \right) - \inf_{\cH} f \right) 
	\leq \dfrac{\alpha -1}{s_{k+1}^{2}} \sum_{i \geq 0} s_{i} \left( f \left( y_{i} \right) - \inf_{\cH} f \right) , \label{prox:est}
\end{align}
where in the last equation we use \eqref{prox-al:sk}. In Theorem \ref{thm:prox1} we have shown $\sum_{i \geq 0} s_{i} \left( f \left( y_{i} \right) - \inf_{\cH} f \right) < + \infty$, which gives the announced convergence rates. The fact that $\left( x_{k} \right) _{k \geq 0}$ converges weakly to an element in $S=\argmin f$ as $k \to + \infty$ follows from the fact that convergence entails ergodic convergence.
\end{proof}

\section{Operator case}
What underlies our study is the first-order in time evolution system taken as a starting point and its convergence properties. In this regard, it is natural to consider also the evolution equation governed by a cocoercive operator. Indeed,  gradients of convex functions and  cocoercive operators are the two basic examples of monotone operators ensuring the convergence of the associated semi-groups of contractions.
We will successively examine a general cocoercive operator, then the additively structured case.
For monotone operators, the convergence rates are expressed in terms of velocities and residuals.

\subsection{General cocoercive case}
Recall that
a single-valued operator $M:\cH \rightarrow \cH$ is called $\rho$-cocoercive (with $\rho>0$) if 

\[
\langle M(x)-M(y),x-y\rangle \geq \rho\norm{M(x)-M(y)}^2\quad \forall x,y\in \cH.
\]

We mention the  following classical properties of cocoercive operators:

\textit{a)} If $M$ is  $\rho$-cocoercive, then it  is monotone and $\frac{1}{\rho}$-Lipschitz continuous, and hence maximally monotone. 

\textit{b)} If $f$ is a convex differentiable function whose gradient is $L$-Lipschitz continuous, then the operator $\nabla f$ is $1/L$-cocoercive (Baillon-Haddad theorem).

Let us come to our study. Given $M \colon \cH \to \cH$ a $\rho$-cocoercive operator, consider the monotone equation
\begin{equation}\label{op 01}
M \left( x \right) = 0 ,
\end{equation}
and the associated evolution equation
\begin{equation}\label{op 02}
\dot{z} \left( t \right) + M \left( z \left( t \right) \right) = 0 .
\end{equation}
We denote the solution set of \eqref{op 01} by $S$,  which we assume to be nonempty.
We recall the results in \cite[Theorems 11 and 16]{BC} which state that $z \left( t \right)$ converges weakly to a point in $S$. In addition
\begin{equation}\label{op 03}
\int_{t_{0}}^{+ \infty} \left\lVert \dot{z} \left( t \right) \right\rVert ^{2} dt = \int_{t_{0}}^{+ \infty} \left\lVert M \left( z \left( t \right) \right) \right\rVert ^{2} dt < + \infty 
\textrm{ and }
\left\lVert M \left( z \left( t \right) \right) \right\rVert = o \left( \dfrac{1}{\sqrt{t}} \right) \textrm{ as } t \to + \infty .
\end{equation}

Further we develop an analysis similar to the one in  Subsection \ref{sec:explicit} in the case of a gradient operator.
Set $y \left( s \right) := z \left( \tau \left( s \right) \right)$, where $\tau \colon \left[ t_{0} , + \infty \right[ \to \mathbb{R}^{+}$ is a continuously differentiable increasing function satisfying $\lim_{s \to + \infty} \tau \left( s \right) = + \infty$.
According to the derivation chain rule, we have
$$
\dot{y} \left( s \right) = \dot{\tau} \left( s \right) \dot{z} \left( \tau \left( s \right) \right)	
$$
which, by setting $t := \tau \left( s \right)$ in \eqref{op 02}, gives
$$
\dot{z} \left( \tau \left( s \right) \right) + M \left( z \left( \tau \left( s \right) \right) \right) = 0 .	
$$
Combining the two relations above gives
\begin{equation}\label{op 04}
\dot{y} \left( s \right) + \dot{\tau} \left( s \right) M \left( y \left( s \right) \right) = 0 .
\end{equation}
Taking the derivative with respect to $s$ of the above equation gives
\begin{equation}\label{op 05}
\ddot{y} \left( s \right) + \ddot{\tau} \left( s \right) M \left( y \left( s \right) \right) + \dot{\tau} \left( s \right) \dfrac{d}{ds} \left( M \left( y \left( s \right) \right) \right) = 0 .
\end{equation}
Multiplying both sides of \eqref{op 04} by $\frac{\alpha}{s}$, then adding the result to \eqref{op 05}, we obtain
\begin{equation}\label{op 06}
\ddot{y} \left( s \right) + \dfrac{\alpha}{s} \dot{y} \left( s \right) + \dot{\tau} \left( s \right) \dfrac{d}{ds} \left( M \left( y \left( s \right) \right) \right) + \left( \ddot{\tau} \left( s \right) + \dfrac{\alpha}{s} \dot{\tau} \left( s \right) \right) M \left( y \left( s \right) \right) = 0 .
\end{equation}
As a subsequent result of \eqref{op 03} we have
\vspace*{-10pt}
\begin{equation*}
\left\lVert M \left( y \left( s \right) \right) \right\rVert = o \left( \dfrac{1}{\sqrt{\tau \left( s \right)}} \right) \textrm{ as } s \to + \infty .
\end{equation*}
Taking
\vspace*{-10pt}
\begin{equation*}
\tau \left( s \right) = \dfrac{s^{2}}{2 \left( \alpha+1 \right)},
\end{equation*}
then \eqref{op 06} becomes
\begin{equation}\label{op 08}
\ddot{y} \left( s \right) + \dfrac{\alpha}{s} \dot{y} \left( s \right) + \dfrac{s}{\alpha+1} \dfrac{d}{ds} \left( M \left( y \left( s \right) \right) \right) + M \left( y \left( s \right) \right) = 0 .
\end{equation}
We obtain the rate
\vspace*{-10pt}
\begin{equation*}
\left\lVert M \left( y \left( s \right) \right) \right\rVert = o \left( \dfrac{1}{s} \right) \textrm{ as } s \to + \infty .
\end{equation*}
Moreover, it holds
\vspace*{-10pt}
\begin{equation*}
\int_{\tau \left( t_{0} \right)}^{+\infty} s \left\lVert M \left( y \left( s \right) \right) \right\rVert ^{2} ds < + \infty .
\end{equation*}

Now let us do the reverse, and start from  a solution  trajectory $y(\cdot)$ of \eqref{op 06} satisfying  $y \left( {s_{0}} \right) := y_{0}$ and $\dot{y} \left( {s_{0}} \right) := y_{1}$. Using similar arguments as in Theorem \ref{Thm-rescale-averaging}, we multiply both sides of \eqref{op 08} by $s^{\alpha}>0$ and get
\begin{equation*}
\dfrac{d}{ds} \left( s^{\alpha} \dot{y} \left( s \right) + \dfrac{s^{\alpha+1}}{\alpha+1} M \left( y \left( s \right) \right) \right) = 0 .
\end{equation*}
This leads to
\begin{equation}\label{op 07}
\dot{y} \left( s \right) + \dfrac{s}{\alpha+1} M \left( y \left( s \right) \right) = \dfrac{c_0}{s^{\alpha}} ,
\end{equation}
where $c_{0} \in \cH$ is a constant that can be determined from the Cauchy condition.
Set $z \left( t \right) := y \left( \sigma \left( t \right) \right)$, where
\begin{equation}\label{op 077}
\sigma \left( t \right) = \sqrt{2 \left( \alpha+1 \right) t} .
\end{equation}
After calculations similar to those made previously, we arrive at
$$
\dot{z} \left( t \right) + M \left( z \left( t \right) \right) = \dfrac{c}{t^{\frac{\alpha+1}{2}}},
$$
where $c \in \cH$ is a constant.
Our study therefore falls within the properties of the perturbed system 
\begin{equation}\label{eq:coer_10}
\dot{z} \left( t \right) + M\left( z \left( t \right) \right) = g(t),
\end{equation}
where, as $\alpha>1$,  the external perturbation  $g \colon \left[ t_{0} , + \infty \right[ \to \cH, g(t) = \frac{c}{t^{\frac{\alpha+1}{2}}} $, is such that
\begin{equation}\label{oppert 01}
\int_{t_{0}}^{+ \infty} \left\lVert g \left( t \right) \right\rVert dt < + \infty \, \textrm{ and } \int_{t_{0}}^{+ \infty} t \left\lVert g \left( t \right) \right\rVert ^{2} dt < + \infty .
\end{equation}

We have the following result.
\begin{theorem}\label{thm:cocoercive}
Let $z \colon \left[ t_{0} , + \infty \right[ \to \cH$ be a solution trajectory of
\begin{equation}
	\dot{z} \left( t \right) + M \left( z \left( t \right) \right) = g(t),
\end{equation}
where $g$ fulfils \eqref{oppert 01}. Then the following statements are true:
\begin{enumerate}
	\item 
	(convergence rate of the operator norm) \quad 
	$
	\left\lVert M \left( z \left( t \right) \right) \right\rVert = o \left( \frac{1}{\sqrt{t}}  \right) \textrm{ as } t \to + \infty ;
	$
	
	\item 
	(integral estimate of the operator norm) \quad
	$
	\int_{t_{0}}^{+\infty} \left\lVert M \left( z \left( t \right) \right) \right\rVert ^{2} dt < +\infty;
	$
	
	\item 
	the solution trajectory $z(t)$ converges weakly as $t \to +\infty$, and its limit belongs to $S = \Zer M$.
\end{enumerate}
\end{theorem}

\noindent Combining the above theorem with the time scaling defined in \eqref{op 077} gives the following result.

\begin{theorem}\label{thm:cocoer_1}
Let $M:\cH \rightarrow \cH$ be a  cocoercive operator and
$y: [s_{0}, +\infty[ \to \cH$ be a solution trajectory of	
\begin{equation}
	\ddot{y} \left( s \right) + \dfrac{\alpha}{s} \dot{y} \left( s \right) + \dfrac{s}{\alpha + 1} \dfrac{d}{ds} \left( M\left( y \left( s \right) \right) \right) + M \left( y \left( s \right) \right) = 0,
\end{equation}
where $\alpha >1$.  Then the following statements are true:
\begin{enumerate}
	\item 
	(convergence rate of the operator norm) \quad
	$
	\left\lVert M\left( y \left( s \right) \right) \right\rVert = o \left( \frac{1}{s} \right) \textrm{ as } s \to + \infty ;
	$
	
	\item 
	(integral estimate of the gradients) \quad
	$
	\int_{s_{0}}^{+\infty} s \left\lVert M \left( y \left( s \right) \right) \right\rVert ^{2} ds < + \infty ;
	$
	
	\item 
	the solution trajectory $y(s)$ converges weakly as $s \to +\infty$, and its limit belongs to $S = \Zer M$.
\end{enumerate}	
%
%
%
\end{theorem}

\begin{remark}
Theorem \ref{thm:cocoer_1} is valid for an arbitrary cocoercive operator. By contrast, the corresponding result without the Newton correction term requires that the coefficient of the cocoercive operator asymptotically tends to infinity, see Attouch and Peypouquet \cite{AP-max}.
\end{remark}

\begin{remark}
Recently, in \cite{BCN} the  closely related dynamics
\begin{equation}\label{variant_BCN}
	\ddot{y} \left( s \right) + \dfrac{\alpha}{s} \dot{y} \left( s \right) + \dfrac{2s}{\alpha+1} \dfrac{d}{ds} \left( M \left( y \left( s \right) \right) \right) + M \left( y \left( s \right) \right) = 0
\end{equation}
has been considered. Compared to \eqref{op 08}, the  difference lies in the coefficient of $\frac{d}{ds} \left( M \left( y \left( s \right) \right) \right)$.
According to \cite[Theorem 4 and Theorem 7]{BCN}, the solution trajectory of \eqref{variant_BCN} verifies
$$
\int_{t_0}^{+\infty}  s^{3} \left\lVert M \left( y \left( s \right) \right) \right\rVert ^{2} ds < + \infty, \, \mbox{ and } \, \left\lVert M \left( y \left( s \right) \right) \right\rVert = o \left( \dfrac{1}{s^{2}} \right) \textrm{ as } s \to + \infty .
$$
It would be interesting to know if we can achieve the same convergence rate using time scaling and perturbation techniques.
\end{remark}

\subsection{Additively structured monotone inclusions}

Our interest now is to solve the structured monotone inclusion problem 
\begin{equation*}
0 \in A(x)+B(x), 
\end{equation*}
where $A$ is maximally monotone, and $B$ is cocoercive with $(A+B)^{-1}(0) \neq \emptyset$. 
A favorable situation occurs when one can compute the resolvent operator of $A$
\[
J_{\mu A} = (I + \mu A )^{-1},  \quad \mu > 0 .
\]
In this case, we can develop a strategy parallel to the one which consists in replacing a maximally monotone operator by its Yosida approximation. 
Indeed, given $\mu>0$, we have
\begin{equation}\label{eq.zero_oper_1}
(A+B) (x) \ni 0 \iff x- J_{\mu A}( x-\mu B(x)) =0\iff M_{A,B,\mu}(x) =0,
\end{equation}
where $M_{A,B,\mu}: \cH \to \cH$ is the single-valued operator defined by 
\begin{equation}\label{Max_Mon_structured}
M_{A,B,\mu}(x) = \frac{1}{\mu} \pa{x- J_{\mu A}( x-\mu B(x))}.
\end{equation}

{The operator} $M_{A,B,\mu}$ is closely tied to the well-known forward-backward fixed point operator. Moreover, when $B=0$, $M_{A,B,\mu}=\frac{1}{\mu} \pa{I - J_{\mu A}}$ which is nothing but the Yosida regularization of $A$ with index $\mu$. As a remarkable property, for the paremeter $\mu$ properly set, the operator $M_{A,B,\mu}$ is cocoercive. This is made precise in the following result,  whose proof relies on the relation between cocoercivity and averaged property.

\begin{proposition} {\rm (\cite[Proposition 4.4]{CY})} \label{cocoer}  Let $A: \cH \to 2^{\cH}$ be a general maximally monotone operator, and  $B:\cH \to \cH $  a monotone operator which is $\lambda$-cocoercive. Assume that $\mu  \in ]0, 2\lambda[ $. Then
$M_{A,B,\mu}$ is $\rho$-cocoercive with 
\vspace{-10pt}
\begin{equation*}
	\rho =\mu\pa{1-\frac{\mu}{4\lambda}}.
\end{equation*}
\end{proposition}
\noindent Therefore, we can develop a strategy parallel to that developed in \cite{AP-max},  which consists in replacing a maximally monotone operator by its Yosida approximation. 
This leads to the inertial dynamics
\begin{equation}\label{eq:cocoer_composite}
\ddot{y} \left( s \right) + \dfrac{\alpha}{s} \dot{y} \left( s \right) + \dfrac{s}{\alpha + 1} \dfrac{d}{ds} \left( M_{A,B,\mu}\left( y \left( s \right) \right) \right) + M_{A,B,\mu} \left( y \left( s \right) \right) = 0 .
\end{equation}
It is immediate to apply Theorem \ref{thm:cocoer_1} to \eqref{eq:cocoer_composite}. This open new perspectives for forward-backward accelerated regularized Newton methods.

Let us illustrate our results by considering the following composite optimization problem
\begin{equation*}
\min_{y \in \R^{n}} \left\lbrace f \left( y \right) := \dfrac{1}{2} \left\lVert Ay - b \right\rVert ^{2} + g \left( y \right) \right\rbrace ,
\end{equation*}
where $A$ is a linear operator from $\R^{n}$ to $\R^{m}$, $g \colon \R^{n} \to \R \cup \left\lbrace + \infty \right\rbrace$ is a proper, convex and lower semicontinuous function which acts as a regularizer. This class of problems covers many interesting situations arising in the signal and image processing such as the LASSO problem. We are in the setting of \eqref{Max_Mon_structured} with
$$
M_{A, B, \mu} \left( y \right) := \dfrac{1}{\mu} \left( y - \prox_{\mu g} \left( y + \mu A^{*} \left( b - Ay \right) \right) \right) ,
$$
where $0 < \mu < \frac{1}{\left\lVert A \right\rVert ^{2}}$.
Following \cite{ACFR}, we can also interpret our problem as minimizing the Moreau envelope of $f$ computed for an ad hoc metric. Precisely, if $W := \frac{1}{\mu} \mathbb{I} - A^{*}A$ and $f_{W} \left( y \right) := \argmin_{u \in \R^{n}} \left\lbrace f \left( u \right) + \frac{1}{2} \left\lVert y - u \right\rVert _{W}^{2} \right\rbrace$, then $\nabla f_{W} \left( y \right) = M_{A, B, \mu} \left( y \right)$ and $\inf_{\cH} f = \inf_{\cH} f_{W}$.
So, we can use this example to illustrate  also the convergence properties of the dynamic with explicit Hessian damping  considered in Subsection \ref{sec:explicit}. Thus, we can  expect a convergence rate of $1/s^{2}$ for $f \left( y_{W}(s) \right) - \inf_{\cH} f$, where $y_{W} \left( s \right) := \prox_{\mu g} \left( y \left( s \right) + \mu A^{*} \left( b - Ay \left( s \right) \right) \right)$.

We conduct the experiment for a randomly generated matrix $A$ and by taking $g$ to be the $\ell_{1}$ norm with various choices for $\alpha \in \left\lbrace 1.001, 2, 3 ,5 \right\rbrace$.  Notice that we do not need need to solve \eqref{eq:cocoer_composite} directly but rather exploit the fact that it can be rewritten as the first order equation
$$
\dot{y} \left( s \right) + \dfrac{s}{\alpha+1} M_{A, B, \mu} \left( y \left( s \right) \right) = \dfrac{s_{0}^{\alpha}}{s^{\alpha}} \left( y_{1} + \dfrac{s_{0}}{\alpha+1} y_{0} \right) ,
$$
as shown previously. We then use the Runge-Kutta method to solve this equation numerically.
Below we plot for the different choices of $\alpha$ the values of $f$ at $y_{W} \left( s \right)$ and of the norm of the operator at $y \left( s \right)$, respectively.
\begin{figure}[!htb]
\minipage{0.48\textwidth}
\includegraphics[width=\linewidth]{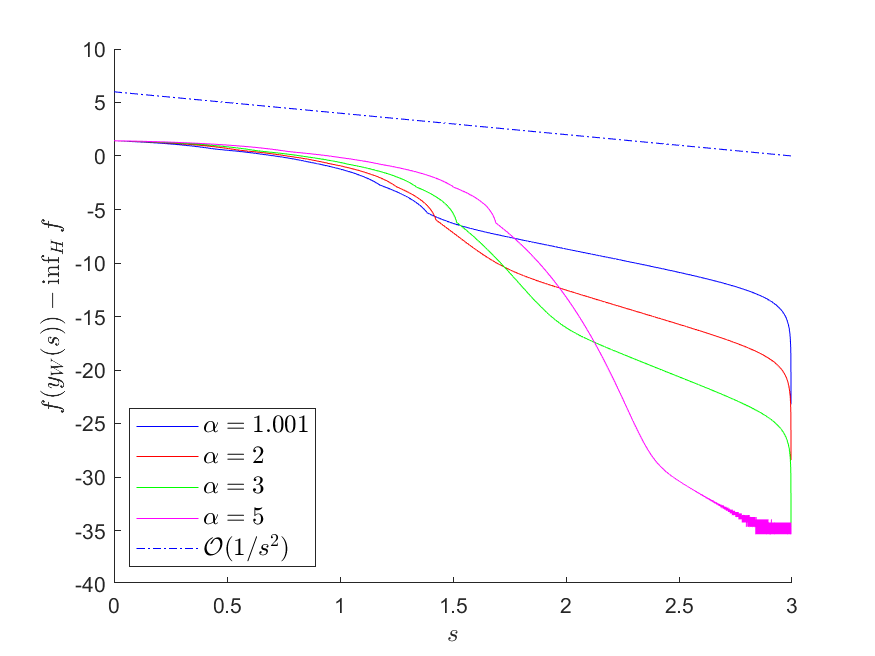}
\caption{Convergence of funtion values in logarithmic scale.}
\label{fig:proxgrad_logval}
\endminipage\hfill
\minipage{0.48\textwidth}
\includegraphics[width=\linewidth]{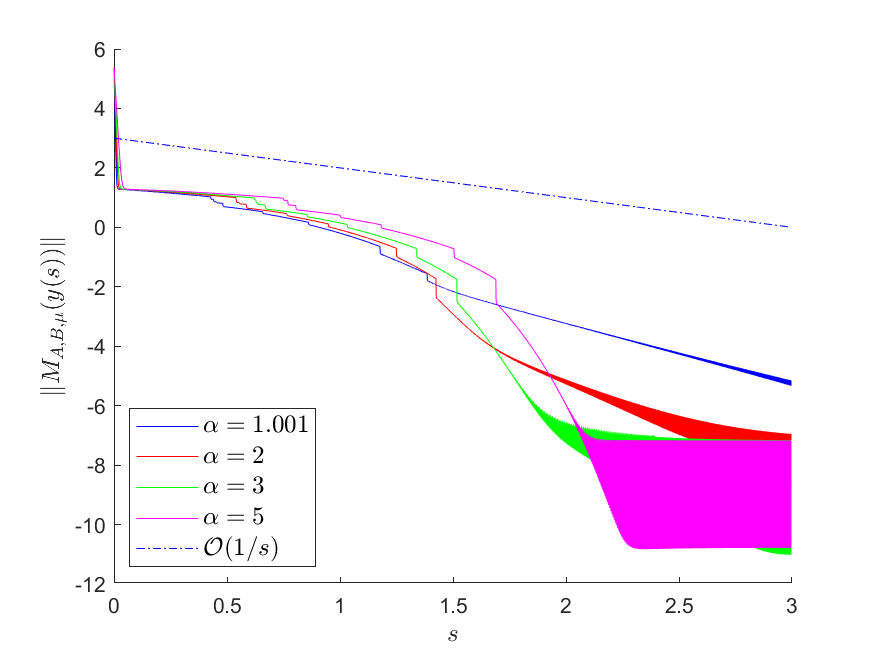}
\caption{Convergence of operator values in logarithmic scale.}
\label{fig:proxgrad_loggrad}
\endminipage
\end{figure}

\section{Conclusion and perspectives}

Our study has shed new light on the acceleration of gradient methods for general convex differentiable optimization.
Based on a dynamical approach, and using relatively elementary tools, namely time scaling and time averaging, we were able to introduce a large class of damped inertial dynamics with proven fast convergence rates.
The method is very flexible and can be adapted to various initial start-up dynamics. We paid particular attention to the steepest continuous descent and the regularized Newton method. Our approach confirms the important role played by the Hessian driven damping role in these questions.
Our results open the door to a systematic study of the corresponding fast proximal and gradient methods.
Importantly, our approach also works when we take a cocoercive operator instead of a gradient operator. Applications to inertial forward-backward methods, ADMM, can therefore be expected, to cite some of the most popular algorithms for convex structured optimization.
Application to stochastic differential equations can be considered as well.

\section{Appendix}

\subsection{Auxiliary results}

Opial's lemma in continuous form  is a basic ingredient of the asymptotic analysis of dynamical systems.
\begin{lemma}\label{Opial} {\rm (Opial)} Let $S$ be a nonempty subset of $\mathcal H$ and let $x:[t_0,+\infty[\to \mathcal H$. Assume that 
\begin{itemize}
	\item [(i)] for every $z\in S$, $\lim_{t\to\infty}\|x(t)-z\|$ exists;
	\item [(ii)] every weak sequential limit point of $x(t)$, as $t\to\infty$, belongs to $S$.
\end{itemize}
Then $x(t)$ converges weakly as $t\to +\infty$, and its limit belongs to $S$.
\end{lemma}
In the analysis of the perturbed dynamical systems Gronwall-Bellman's lemma (see \cite[Lemme A.5]{breziseqevo}) plays an important role.
\begin{lemma}\label{lem GB}
Let {$T \geq \delta \geq 0$ and} $g \colon \left[ \delta , T \right] \to \mathbb{R}^{+}$ be integrable, and let $c \geq 0$. Suppose that $h \colon \left[ \delta , T \right] \to \mathbb{R}$ is continuous and
\vspace*{-10pt}
\begin{equation*}
	\dfrac{1}{2} h^{2} \left( t \right) \leq \dfrac{1}{2} c^{2} + \int_{\delta}^{t} h \left( t \right) g \left( t \right) dt 
\end{equation*}
for every $t \in \left[ \delta , T \right]$. Then $\left\lvert h \left( t \right) \right\rvert \leq c + \int_{\delta}^{t} g \left( u \right) du$ for every $t \in \left[ \delta , T \right]$.
\end{lemma}

\begin{lemma}\label{lem o(1/t)}
Let {$\delta \geq 0$ and} $h \colon \left[ \delta , + \infty \right] \to \R^{+}$ be a nonincreasing function  belonging to $\mathbb{L}^{1} \left( \left[ \delta , + \infty \right] \right)$. It  holds $\lim_{t \to + \infty} th \left( t \right) = 0$.
\end{lemma}
\begin{proof}
The nonincreasing property of $h$ implies that $\frac{d}{dt} \left( th \left( t \right) \right) = h \left( t \right) + t \dot{h} \left( t \right) \leq h \left( t \right)$ for every $t \geq t_0$.  Since $h \in \mathbb{L}^{1} \left( \left[ \delta , + \infty \right] \right)$, the result follows from \cite[Lemma 5.2]{AAS}.
\end{proof}

\begin{lemma}\label{lem lim-0}
Let $\delta \geq 0$ and $a \colon \left[ \delta , + \infty \right] \to \R^{+}$ be a positive real valued function such that $\lim_{s \to + \infty} a(s)=0$.  Then $\lim_{s \to + \infty} A(s)=0$, where for $p>0$
$$
A(s) =   \frac{1}{s^{p}}\int_{s_{0}}^s u^{p-1} a(u)du.
$$
\end{lemma}
\begin{proof}
Given $\epsilon >0$,  let $T_{\epsilon}$ such that $s_{0} <T_{\epsilon} $ and $
a(u) \leq \epsilon$ for $t \geq T_{\epsilon}$.
For $s > T_{\epsilon}$, we have
\begin{eqnarray*}
	A(s) = \frac{1}{s^{p}}\int_{s_{0}}^{T_{\epsilon}} u^{p-1} a(u)du + \frac{1}{s^{p}}\int_{T_{\epsilon}}^s u^{p-1} a(u)du \leq  \frac{1}{s^{p}}\int_{s_{0}}^{T_{\epsilon}} u^{p-1} a(u)du + \epsilon .
\end{eqnarray*}
Letting $s$ converge to $+\infty$,  it yields
\begin{equation*}
	\limsup_{s\to +\infty } A(s) \leq \epsilon.
\end{equation*}
This being true for any $\epsilon >0$, we infer that $\lim_{s \to + \infty} A(s)=0$, which gives the claim. 
\end{proof}

\subsection{Missing proofs}

\begin{proof}[Proof of Theorem \ref{theorem SD pert}]
For sake of completeness, let us recall some of the arguments used in the asymptotic analysis of the perturbed steepest descent system when the perturbation $g$ satisfies \eqref{pert 01}.
Given $z_{*} \in S$, let $T > t_{0}$ be fixed and for every $T \geq t \geq t_0$ consider
\vspace*{-10pt}
\begin{equation*}
	\cE_{T} \left( t \right) := t \left( f \left( z \left( t \right) \right) - \inf_{\cH} f \right) + \dfrac{1}{2} \left\lVert z \left( t \right) - z_{*} \right\rVert ^{2} + \int_{t}^{T} \left\langle z \left( u \right) - z_{*} + \tau \dot{z} \left( u \right) , g \left( u \right) \right\rangle du .
\end{equation*}
Differentiating  $\cE_{T}$ gives for every $T \geq t \geq t_0$ 
\begin{align*}
	\dfrac{d}{dt} \cE_{T} \left( t \right) 
	& = f \left( z \left( t \right) \right) - \inf_{\cH} f + t \left\langle \nabla f \left( z \left( t \right) \right) , \dot{z} \left( t \right) \right\rangle + \left\langle z \left( t \right) - z_{*} , \dot{z} \left( t \right) - g \left( t \right) \right\rangle - t \left\langle \dot{z} \left( t \right) , g \left( t \right) \right\rangle \nonumber \\
	& = f \left( z \left( t \right) \right) - \inf_{\cH} f - t \left\lVert \dot{z} \left( t \right) \right\rVert ^{2} - \left\langle z \left( t \right) - z_{*} , \nabla f \left( z \left( t \right) \right) \right\rangle \nonumber \\
	& \leq - t \left\lVert \dot{z} \left( t \right) \right\rVert ^{2} , 
\end{align*}
where the second equality comes from \eqref{pert SD} and the last inequality follows from the convexity of $f$. 
By integration from $t_{0}$ to $t$ and by denoting ${C_3}:=t_0 \left( f \left( z \left( t_0 \right) \right) - \inf_{\cH} f \right) + \dfrac{1}{2} \left\lVert z \left( t_o \right) - z_{*} \right\rVert ^{2} \geq 0$  we deduce that 
\begin{align}
	& t \left( f \left( z \left( t \right) \right) - \inf_{\cH} f \right) + \dfrac{1}{2} \left\lVert z \left( t \right) - z_{*} \right\rVert ^{2} \nonumber \\
	\leq \ 	& {C_3} + \int_{t_{0}}^{t} \left\langle z \left( u \right) - z_{*} , g \left( u \right) \right\rangle du + \int_{t_{0}}^{t} \tau \left\langle \dot{z} \left( u \right) , g \left( u \right) \right\rangle du - \int_{t_{0}}^{t} \tau \left\lVert \dot{z} \left( u \right) \right\rVert ^{2} du \nonumber \\
	\leq \ 	& {C_3} + \int_{t_{0}}^{t} \left\lVert z \left( u \right) - z_{*} \right\rVert \left\lVert g \left( u \right) \right\rVert du + \dfrac{1}{2} \int_{t_{0}}^{+ \infty} \tau \left\lVert g \left( u \right) \right\rVert ^{2} du - \dfrac{1}{2} \int_{t_{0}}^{t} \tau \left\lVert \dot{z} \left( u \right) \right\rVert ^{2} du \label{SD_Lyap_1}.
\end{align}
We obtain that there exists ${C_4} \geq 0$ such that the following estimate is true for every  $t \geq t_{0}$
\vspace*{-5pt}
\begin{equation*}
	\dfrac{1}{2} \left\lVert z \left( t \right) - z_{*} \right\rVert ^{2} \leq {C_4} + \int_{t_{0}}^{t} \left\lVert z \left( u \right) - z_{*} \right\rVert \left\lVert g \left( u \right) \right\rVert du .
\end{equation*}
According to Gronwall's lemma (see Lemma \ref{lem GB} in the Appendix), we conclude that for every $t \geq t_{0}$
$$
\left\lVert z \left(t \right) - z_{*} \right\rVert \leq \sqrt{2 {C_4}} + \int_{t_{0}}^{t} \left\lVert g \left( u \right) \right\rVert du \leq \sqrt{2 {C_4}} + \int_{t_{0}}^{+ \infty} \left\lVert g \left( u \right) \right\rVert du < + \infty .
$$
The trajectory is therefore bounded. Using this property and condition \eqref{pert 01} allow us to assert from \eqref{SD_Lyap_1} that there exists ${C_5} \geq 0$ such that for every $t \geq t_0$
\begin{equation}\label{SD_Lyap_2}
	t \left( f \left( z \left( t \right) \right) - \inf_{\cH} f \right) + \dfrac{1}{2} \left\lVert z \left( t \right) - z_{*} \right\rVert ^{2} + \dfrac{1}{2} \int_{t_{0}}^{t} \tau \left\lVert \dot{z} \left( u \right) \right\rVert ^{2} du \leq {C_5}.
\end{equation}
%
Since $t \left( f \left( z \left( t \right) \right) - \inf_{\cH} f \right) + \frac{1}{2} \left\lVert z \left( t \right) - z_{*} \right\rVert ^{2} \geq 0$ for every $t \geq t_{0}$,  letting $t$ tend to $+ \infty$ in \eqref{SD_Lyap_2} we deduce that
\begin{equation*}
	\int_{t_0}^{+\infty} t \left\lVert \dot{z} \left( t \right) \right\rVert ^{2} dt < + \infty .
\end{equation*}
According to the constitutive equation \eqref{pert SD} we get {item \ref{SD pert:int-grad}}
\begin{equation*}
	\int_{t_0}^{+\infty} t \left\lVert \nabla f (z(t) \right\rVert ^{2} dt \leq 2\int_{t_0}^{+\infty} t \left\lVert \dot{z} \left( t \right) \right\rVert ^{2} dt + 2\int_{t_{0}}^{+ \infty} t \left\lVert g \left( t \right) \right\rVert ^{2} dt < +\infty.
\end{equation*}
Differentiate the anchor function, it yields for every $t \geq t_{0}$
\begin{align}
	\dfrac{d}{dt} \left( \dfrac{1}{2} \left\lVert z \left( t \right) - z_{*} \right\rVert ^{2} \right) & = \left\langle z \left( t \right) - z_{*} , \dot{z} \left( t \right) \right\rangle \nonumber \\
	& = - \left\langle z \left( t \right) - z_{*} , \nabla f \left( z \left( t \right) \right) \right\rangle - \left\langle z \left( t \right) - z_{*} , g \left( t \right) \right\rangle \nonumber \\
	& \leq - \left( f \left( z \left( t \right) \right) - \inf_{\cH} f \right) + \sup_{{u} \geq t_{0}} \left\lVert z \left( {u} \right) - z_{*} \right\rVert \cdot \left\lVert g \left( t \right) \right\rVert \label{pert 03} \\
	& \leq \sup_{{u} \geq t_{0}} \left\lVert z \left( {u} \right) - z_{*} \right\rVert \cdot \left\lVert g \left( t \right) \right\rVert . \label{pert 04}
\end{align}
Recall that the trajectory $z(\cdot)$ is bounded. Now let us show that in fact
\begin{equation*}
	\lim_{t \to + \infty} t \left( f \left( z \left( t \right) \right) - \inf_{\cH} f \right) = 0 ,
\end{equation*}
meaning that the convergence rate of $f \left( z \left( t \right) \right) - \inf_{\cH} f$ is actually $o \left( 1/t \right)$.  To this end we integrate \eqref{pert 03} from $t_{0}$ to $t>t_{0}$ and let then $t$ converge to $+ \infty$. This yields \ref{SD pert:int-fun}
\begin{equation}\label{pert 044}
	\int_{t_{0}}^{+ \infty} \left( f \left( z \left( t \right) \right) - \inf_{\cH} f \right) dt = \int_{t_{0}}^{+ \infty} \dfrac{1}{t} t \left( f \left( z \left( t \right) \right) - \inf_{\cH} f \right) dt < + \infty
\end{equation}
and thus $\liminf_{t \to + \infty} t \left( f \left( z \left( t \right) \right) - \inf_{\cH} f \right) = 0$. It remains to show that the limit  exists. To this end we compute the time derivative of $t \left( f \left( z \left( t \right) \right) - \inf_{\cH} f \right)$ for every $t \geq t_{0}$
\begin{align*}
	\dfrac{d}{dt} \left( t \left( f \left( z \left( t \right) \right) - \inf_{\cH} f \right) \right) & = f \left( z \left( t \right) \right) - \inf_{\cH} f + t \left\langle \nabla f \left( z \left( t \right) \right) , \dot{z} \left( t \right) \right\rangle \nonumber \\
	& = f \left( z \left( t \right) \right) - \inf_{\cH} f - t \left\lVert \dot{z} \left( t \right) \right\rVert ^{2} - t \left\langle g \left( t \right) , \dot{z} \left( t \right) \right\rangle \leq f \left( z \left( t \right) \right) - \inf_{\cH} f + \dfrac{1}{4} t \left\lVert g \left( t \right) \right\rVert ^{2}
\end{align*}
and apply once again \cite[Lemma 5.1]{AAS}. Statement \ref{SD pert:rate-fun} follows from  assumption \eqref{pert 01} and \eqref{pert 044}.

\noindent
According to  assumption \eqref{pert 01}, we deduce that the right hand side of \eqref{pert 04} belongs to $\mathbb{L}^{1} \left( \left[ t_{0} , + \infty \right[ \right)$. Therefore, from \cite[Lemma 5.1]{AAS} we obtain that $\lim_{t \to + \infty} \left\lVert z \left( t \right) - z_{*} \right\rVert ^{2} \in \mathbb{R}$ exists, and so $\lim_{t \to + \infty} \left\lVert z \left( t \right) - z_{*} \right\rVert \in \mathbb{R}$ does. In other words, the first condition of Opial's lemma (see Lemma \ref{Opial} in the Appendix) is fulfilled. Furthermore, since $\lim_{t \to + \infty} f \left( z \left( t \right) \right) = \inf_{\cH} f$ and $f$ is convex and weakly lower semicontinuous, the second condition of Opial's lemma is also fulfilled. This gives the weak convergence of the trajectory $z(t)$ as $t \to +\infty$ to an element in $S=\argmin f$.

The proof of \ref{SD pert:rate-grad} follows similarly. First, the assertion \ref{SD pert:int-grad} gives $\liminf_{t \to + \infty} t^{2} \left\lVert \nabla f \left( z \left( t \right) \right) \right\rVert ^{2} = 0$.
Let $L$ be the Lipschitz constant of  $\nabla f$ on a ball  containing the trajectory  $z(\cdot)$. We have for almost every $t \geq t_{0}$
\begin{align*}
	\dfrac{d}{dt} \left( t^{2} \left\lVert \nabla f \left( z \left( t \right) \right) \right\rVert ^{2} \right) & = 2t \left\lVert \nabla f \left( z \left( t \right) \right) \right\rVert ^{2} + t^{2} \left\langle \nabla f \left( z \left( t \right) \right) , \dfrac{d}{dt} \nabla f \left( z \left( t \right) \right) \right\rangle \nonumber \\
	& = 2t \left\lVert \nabla f \left( z \left( t \right) \right) \right\rVert ^{2} - t^{2} \left\langle \dot{z} \left( t \right) , \dfrac{d}{dt} \nabla f \left( z \left( t \right) \right) \right\rangle + t^{2} \left\langle g \left( t \right) , \dfrac{d}{dt} \nabla f \left( z \left( t \right) \right) \right\rangle \nonumber \\
	& \leq 2t \left\lVert \nabla f \left( z \left( t \right) \right) \right\rVert ^{2} - \dfrac{t^{2}}{L} \left\lVert \dfrac{d}{dt} \nabla f \left( z \left( t \right) \right) \right\rVert ^{2} + 4Lt^{2} \left\lVert g \left( t \right) \right\rVert ^{2} + \dfrac{t^{2}}{L} \left\lVert \dfrac{d}{dt} \nabla f \left( z \left( t \right) \right) \right\rVert ^{2} \nonumber \\
	& = 2t \left\lVert \nabla f \left( z \left( t \right) \right) \right\rVert ^{2} + 4Lt^{2} \left\lVert g \left( t \right) \right\rVert ^{2} . \nonumber
\end{align*}
The convergence rate follows by applying once again \cite[Lemma 5.1]{AAS} and by using that $t \mapsto tg(t) \in \mathbb{L}^{2} \left( \left[ t_{0} , + \infty \right[ \right)$.
\end{proof}

\begin{proof}[Proof of Theorem \ref{Theorem regNewton}]
Given $z_{*} \in S$, we define the following function for every $t \geq t_{0}$, which is nonnegative due to the convexity of $f$ and the fact that $v \left( t \right) = \nabla f \left( z \left( t \right) \right)$
\vspace*{-10pt}
\begin{equation*}
	\phi \left( t \right) := \left[ f \left( z_{*} \right) - f \left( z \left( t \right) \right) - \left\langle v \left( t \right) , z_{*} - z \left( t \right) \right\rangle \right] + \dfrac{\lambda}{2} \left\lVert z_{*} - z \left( t \right) \right\rVert ^{2} .
\end{equation*}
According to \eqref{eq:Dyn_Newton_1} and the convexity of $f$ we have for every $t \geq t_0$
\begin{align*}
	\dfrac{d}{dt} \phi \left( t \right) & = - \left\langle v \left( t \right) , \dot{z} \left( t \right) \right\rangle - \left\langle \dot{v} \left( t \right) , z_{*} - z \left( t \right) \right\rangle + \left\langle v \left( t \right) , \dot{z} \left( t \right) \right\rangle - \lambda \left\langle z_{*} - z \left( t \right) , \dot{z} \left( t \right) \right\rangle \nonumber \\
	& = \left\langle v \left( t \right) , z_{*} - z \left( t \right) \right\rangle \leq \inf_{\cH} f - f \left( z \left( t \right) \right) \leq 0 .
\end{align*}
We integrate from $t_{0}$ to $t>t_{0}$ and let $t$ tend to $+\infty$, and obtain \ref{regNewton:int-fun}.

The inequality above also implies that the function $\phi$ is nonincreasing on $\left[ t_{0} , + \infty \right[$, from which we deduce that $\phi(t)$ converges as $t \to +\infty$, and that the trajectory $z(\cdot)$ is bounded.
Let $L>0$ be the Lipschitz constant of $\nabla f$ on a ball that contains the trajectory $z(\cdot)$. It follows from \eqref{eq:Dyn_Newton_1} that for every $t \geq t_0$
\begin{align*}
	\dfrac{d}{dt} \left( f \left( z \left( t \right) \right) - \inf_{\cH} f \right) & = \left\langle v \left( t \right) , \dot{z} \left( t \right) \right\rangle = - \lambda \left\lVert \dot{z} \left( t \right) \right\rVert ^{2} - \left\langle \dot{v} \left( t \right) , \dot{z} \left( t \right) \right\rangle = - \lambda \left\lVert \dot{z} \left( t \right) \right\rVert ^{2} - \left\langle \dfrac{d}{dt} \left( \nabla f \left( z \left( t \right) \right) \right) , \dot{z} \left( t \right) \right\rangle \nonumber \\
	& \leq - \lambda \left\lVert \dot{z} \left( t \right) \right\rVert ^{2} - \dfrac{1}{L} \left\lVert \dfrac{d}{dt} \left( \nabla f \left( z \left( t \right) \right) \right) \right\rVert ^{2} = - \lambda \left\lVert \dot{z} \left( t \right) \right\rVert ^{2} - \dfrac{1}{L} \left\lVert \dot{v} \left( t \right) \right\rVert ^{2} \leq 0 .
\end{align*}
Since $t \mapsto f \left( z \left( t \right) \right) - \inf_{\cH} f$ is nonincreasing, by applying Lemma \ref{lem o(1/t)} in the Appendix we obtain \ref{regNewton:rate-fun}. 

In addition, we derive that
\begin{align*}
	\dfrac{d}{dt} \left( t \left( f \left( z \left( t \right) \right) - \inf_{\cH} f \right) \right) & = f \left( z \left( t \right) \right) - \inf_{\cH} f + \dfrac{d}{dt} \left( f \left( z \left( t \right) \right) - \inf_{\cH} f \right) \nonumber \\
	& \leq f \left( z \left( t \right) \right) - \inf_{\cH} f - \lambda t \left\lVert \dot{z} \left( t \right) \right\rVert ^{2} - \dfrac{1}{L} t \left\lVert \dot{v} \left( t \right) \right\rVert ^{2}
\end{align*}

Integrating again from $t_{0}$ to $t>t_{0}$ and letting $t$ tend to $+\infty$, we obtain \ref{regNewton:int-vel-grad}. Indeed, for every $t \geq t_{0}$ we have $\left\lVert v \left( t \right) \right\rVert ^{2} \leq 2 \lambda^2 \left\lVert \dot{z} \left( t \right) \right\rVert ^{2} + 2 \left\lVert \dot{v} \left( t \right) \right\rVert ^{2}$, 
therefore
\begin{equation}\label{eq:Dyn_Newton_5}
	\int_{t_{0}}^{+ \infty} t \left\lVert v \left( t \right) \right\rVert ^{2} dt = \int_{t_{0}}^{+ \infty} t \left\lVert \nabla f \left( z \left( t \right) \right) \right\rVert ^{2} dt < +\infty .
\end{equation}	

Consequently, for item $iv)$ we can use again Lemma \ref{lem o(1/t)}. Precisely, observe that for every $t \geq t_0$
\begin{equation*}
	\dfrac{d}{dt} \left( \dfrac{1}{2} \left\lVert v \left( t \right) \right\rVert ^{2} \right) = \left\langle v \left( t \right) , \dot{v} \left( t \right) \right\rangle = - \lambda \left\langle \dot{{z}} \left( t \right) , \dot{v} \left( t \right) \right\rangle - \left\lVert \dot{v} \left( t \right) \right\rVert ^{2} \leq 0 ,
\end{equation*}
which means $t \mapsto \frac{1}{2} \left\lVert v \left( t \right) \right\rVert ^{2}$ is nonincreasing. The conclusion follows by using \eqref{eq:Dyn_Newton_5}  and by noticing that
\begin{equation*}
	\left\lVert v \left( t \right) \right\rVert ^{2} = \left\lVert \lambda \dot{z} \left( t \right) + \dot{v} \left( t \right) \right\rVert ^{2} = \lambda^{2} \left\lVert  \dot{z} \left( t \right) \right\rVert ^{2} + \left\lVert \dot{v} \left( t \right) \right\rVert ^{2} + 2 \lambda \left\langle \dot{z} \left( t \right) , \dot{v} \left( t \right) \right\rangle \geq \lambda^{2} \left\lVert  \dot{z} \left( t \right) \right\rVert ^{2} + \left\lVert \dot{v} \left( t \right) \right\rVert ^{2} .
\end{equation*}

Finally, let us recall that $\lim_{t \to + \infty} \phi \left( t \right) \in \R$ exists. Moreover, according to \ref{regNewton:rate-fun}, \ref{regNewton:rate-grad-vel}, and the fact that the trajectory $z(\cdot)$ is bounded, we conclude that 
\begin{equation*}
	\lim_{t \to + \infty} \dfrac{\lambda}{2} \left\lVert z_{*} - z \left( t \right) \right\rVert ^{2} = \lim_{t \to + \infty} \phi \left( t \right) \in \R .
\end{equation*}
This yields the first condition of Opial's lemma.
Now let $\widehat{z}$ be a weak sequential cluster point of the trajectory. Since $v \left( t \right) = \nabla f \left( z \left( t \right) \right)$ for every $t \geq t_{0}$, it follows from the sequential closedness property of $\nabla f$ in the weak $\times$ strong topology on $\cH \times \cH$ that $\widehat{z} \in S$, which proves that the second condition of Opial's lemma is also satisfied. This completes the proof.
\end{proof}

\begin{proof}[Proof of Theorem \ref{thm:cocoercive}]
For the sake of completeness, we briefly sketch the Lyapunov analysis of \eqref{eq:coer_10}.
Given $z_{*} \in S$, let us first fix $T > t_{0}$,  and for every $T \geq t \geq t_{0}$ consider
$$
\cE_{T} \left( t \right) := \dfrac{1}{2} \left\lVert z \left( t \right) - z_{*} \right\rVert ^{2} + \int_{t}^{T} \left\langle z \left( u \right) - z_{*} , g \left( u \right) \right\rangle du .
$$
For every $T \geq t \geq t_{0}$ it holds
\begin{align*}
	\dfrac{d}{dt} \cE_{T} \left( t \right) & = \left\langle z \left( t \right) - z_{*} , \dot{z} \left( t \right) \right\rangle - \left\langle z \left( t \right) - z_{*} , g \left( t \right) \right\rangle = - \left\langle z \left( t \right) - z_{*} , M \left( z \left( t \right) \right) \right\rangle .
\end{align*}
By integration from $t_{0}$ to $t$, we get for $C_6:= \dfrac{1}{2} \left\lVert z \left( t_0 \right) - z_{*} \right\rVert ^{2}$
\begin{align}
	\dfrac{1}{2} \left\lVert z \left( t \right) - z_{*} \right\rVert ^{2} + \int_{t_{0}}^{t} \left\langle z \left( t \right) - z_{*} , M \left( z \left( t \right) \right) \right\rangle du & \leq C_6 + \int_{t_{0}}^{t} \left\langle z \left( u \right) - z_{*} , g \left( u \right) \right\rangle du \nonumber \\
	& \leq C_6 + \int_{t_{0}}^{t} \left\lVert z \left( u \right) - z_{*} \right\rVert \left\lVert g \left( u \right) \right\rVert du . \label{oppert 02}
\end{align}
By Gronwall's lemma (see Lemma \ref{lem GB} in the Appendix), we conclude that for every $t \geq t_{0}$
\vspace{-10pt}
\begin{equation*}
	\left\lVert z \left( t \right) - z_{*} \right\rVert \leq \sqrt{2C_6} + \int_{t_{0}}^{t} \left\lVert g \left( u \right) \right\rVert du \leq \sqrt{2C_6} + \int_{t_{0}}^{+ \infty} \left\lVert g \left( u \right) \right\rVert du < + \infty .
\end{equation*}
This gives the boundedness of the trajectory. Using this assertion in \eqref{oppert 02}, we deduce that for every $t \geq t_0$
\begin{align*}
	\int_{t_{0}}^{t} \left\langle z \left( u \right) - z_{*} , M \left( z \left( u \right) \right) \right\rangle du & \leq C_6 + \sup_{t \geq t_{0}} \left\lVert z \left( u \right) - z_{*} \right\rVert \int_{t_{0}}^{t} \left\lVert g \left( u \right) \right\rVert du \nonumber \\
	& \leq C_6 + \sup_{t \geq t_{0}} \left\lVert z \left( u \right) - z_{*} \right\rVert \int_{t_{0}}^{+ \infty} \left\lVert g \left( u \right) \right\rVert du < + \infty .
\end{align*}
According to the $\rho$-cocoercivity of $M$, and by letting $t$ go to infinity, we obtain
\begin{equation}\label{eq:cocoer_100}
	\rho \int_{t_{0}}^{+ \infty} \left\lVert M \left( z \left( t \right) \right) \right\rVert ^{2} dt \leq \int_{t_{0}}^{+ \infty} \left\langle z \left( t \right) - z_{*} , M \left( z \left( t \right) \right) \right\rangle dt < + \infty.
\end{equation}
Therefore, $\liminf_{t \to + \infty} t \left\lVert M \left( z \left( t \right) \right) \right\rVert ^{2} = 0$.
Let us now compute the time derivative for every $t \geq t_0$
\begin{align*}
	\dfrac{d}{dt} \left( \dfrac{1}{2} t \left\lVert M \left( z \left( t \right) \right) \right\rVert ^{2} \right) & = \dfrac{1}{2}\left\lVert M \left( z \left( t \right) \right) \right\rVert ^{2} + t \left\langle M \left( z \left( t \right) \right) , \dfrac{d}{dt} \left( M \left( z \left( t \right) \right) \right) \right\rangle \nonumber \\
	& = \dfrac{1}{2}\left\lVert M \left( z \left( t \right) \right) \right\rVert ^{2} - t \left\langle \dot{z} \left( t \right) , \dfrac{d}{dt} \left( M \left( z \left( t \right) \right) \right) \right\rangle + t \left\langle g \left( t \right) , \dfrac{d}{dt} \left( M \left( z \left( t \right) \right) \right) \right\rangle .
\end{align*}

Using the cocoercivity of $M$ and the Cauchy-Schwarz inequality, we deduce that
\begin{align*}
	- t \left\langle \dot{z} \left( t \right) , \dfrac{d}{dt} \left( M \left( z \left( t \right) \right) \right) \right\rangle + t \left\langle g \left( t \right) , \dfrac{d}{dt} \left( M \left( z \left( t \right) \right) \right) \right\rangle
	& \leq - \rho t \left\lVert \dfrac{d}{dt} \left( M \left( z \left( t \right) \right) \right) \right\rVert ^{2} + t \left\lVert g \left( t \right) \right\rVert \left\lVert \dfrac{d}{dt} \left( M \left( z \left( t \right) \right) \right) \right\rVert \nonumber \\
	& \leq \dfrac{1}{4 \rho} t \left\lVert g \left( t \right) \right\rVert ^{2} .
\end{align*}

The right hand side of the above inequality belongs to $\mathbb{L}^{1} \left( \left[ t_{0} , + \infty \right[ \right)$ thanks to \eqref{oppert 01} and \eqref{eq:cocoer_100}. According to \cite[Lemma 5.1]{AAS} we have that $\lim_{t \to + \infty} t \left\lVert M \left( z \left( t \right) \right) \right\rVert ^{2} \in \mathbb{R}$ exists, and therefore it must be equal to $0$
$$\left\lVert M \left( z \left( t \right) \right) \right\rVert = o \left( \dfrac{1}{\sqrt{t}} \right) \textrm{ as } t \to + \infty .
$$
This also means that $\lim_{t\to+\infty} M \left( z \left( t \right) \right) =0$. According to the sequential closedness property of the graph of $M$ in the weak $\times$ strong topology on $\cH \times \cH$,  every weak limit point of the trajectory $z$ belongs to $S$.
It remains to verify the first condition of Opial's lemma.
Given $z_{*} \in S$, observe that for every $t \geq t_0$ it holds
\begin{align*}
	\dfrac{d}{dt} \left( \dfrac{1}{2} \left\lVert z \left( t \right) - z_{*} \right\rVert ^{2} \right) & = \left\langle z \left( t \right) - z_{*} , \dot{z} \left( t \right) \right\rangle = - \left\langle z \left( t \right) - z_{*} , M \left( z \left( t \right) \right) \right\rangle + \left\langle z \left( t \right) - z_{*} , g \left( t \right) \right\rangle \nonumber \\
	& \leq \sup_{t \geq t_{0}} \left\lVert z \left( t \right) - z_{*} \right\rVert \cdot \left\lVert g \left( t \right) \right\rVert .
\end{align*}
This shows that the limit $\lim_{t \to + \infty} \left\lVert z \left( t \right) - z_{*} \right\rVert \in \mathbb{R}$ thanks to \eqref{oppert 01} and \cite[Lemma 5.1]{AAS}.
\end{proof}


\bibliographystyle{elsarticle-harv}
\biboptions{sort&compress}
\bibliography{ABN_scaling_aver_references}







\end{document}